\documentclass[11pt,oneside]{amsart}
\usepackage{amsopn}
\usepackage{amsmath,amsthm,amssymb}
\usepackage{prettyref}

\newcommand{\nc}{\newcommand}

 \nc{\aff}{\mathfrak{aff} } \nc{\bb}{\mathfrak{b} }
\nc{\cc}{\mathfrak{c} }  \nc{\dd}{\mathfrak{d} }
 \nc{\ggo}{\mathfrak{g} }
 \nc{\hh}{\mathfrak{h} }  \nc{\ii}{\mathfrak{i} }
 \nc{\jj}{\mathfrak{j} }  \nc{\kk}{\mathfrak{k} }
\nc{\mm}{\mathfrak{m} }   \nc{\nn}{\mathfrak{n} }
\nc{\pp}{\mathfrak{p} }  \nc{\rr}{\mathfrak{r} } \nc{\sg}{\mathfrak{s} }
 \nc{\sog}{\mathfrak{so} }  \nc{\spg}{\mathfrak{sp} }
 \nc{\sug}{\mathfrak{su} }  \nc{\slg}{\mathfrak{sl} }
 \nc{\tg}{\mathfrak{t} }  \nc{\uu}{\mathfrak{u} }
 \nc{\vv}{\mathfrak{v} } \nc{\ww}{\mathfrak{w} }
 \nc{\zz}{\mathfrak{z} }

 \nc{\ggob}{\overline{\mathfrak{g}}}

\nc{\glg}{\mathfrak{gl} }

\nc{\pca}{\mathcal{P}} \nc{\nca}{\mathcal{N}}

 \nc{\vp}{\varphi} \nc{\ddt}{\frac{{\rm d}}{{\rm d}t}}
 \nc{\la}{\langle} \nc{\ra}{\rangle}

 \nc{\SO}{{\sf SO}} \nc{\Spe}{{\sf Sp}} \nc{\Sl}{{\sf Sl}}
 \nc{\SU}{{\sf SU}} \nc{\Or}{{\sf O}} \nc{\U}{{\sf U}}
 \nc{\Gl}{{\sf Gl}} \nc{\Se}{{\sf S}} \nc{\Cl}{{\sf Cl}}
 \nc{\Spin}{{\sf Spin}} \nc{\Pin}{{\sf Pin}}

 \nc{\RR}{{\mathbb R}} \nc{\HH}{{\mathbb H}} \nc{\CC}{{\mathbb C}}
 \nc{\ZZ}{{\mathbb Z}} \nc{\FF}{{\mathbb F}} \nc{\NN}{{\mathbb N}}
 \nc{\GG}{{\mathbb G}} \nc{\JJ}{{\mathbb J}} \nc{\II}{{\mathbb I}}
 \nc{\KK}{{\mathbb K}} \nc{\DD}{{\mathbb D}}

 \nc{\ad}{\operatorname{ad}} \nc{\Ad}{\operatorname{Ad}}
 \nc{\coad}{\operatorname{coad}} \nc{\ct}{\operatorname{T}}
 \nc{\rank}{\operatorname{rank}} \nc{\Irr}{\operatorname{Irr}}
 \nc{\End}{\operatorname{End}} \nc{\Aut}{\operatorname{Aut}}
 \nc{\Inn}{\operatorname{Inn}} \nc{\Der}{\operatorname{Der}}
 \nc{\Dera}{\operatorname{Dera}} \nc{\Auto}{\operatorname{Auto}}
 \nc{\GL}{\operatorname{GL}}
 \nc{\SL}{\operatorname{SL}}

\newenvironment{skproof}{\noindent \emph{Proof.}}{\hfill $\blacksquare$}
 \renewenvironment{proof}{\noindent \emph{Proof.}}{\hfill $\blacksquare$}
 \newenvironment{lmproof}{\noindent \emph{Proof of  Lemma}}{\hfill $\blacksquare$}

\newcommand{\wtV}{\widetilde V }

 \newcommand{\R}{\mathbb R}
\newcommand{\N}{\mathbb N}
\newcommand{\Z}{\mathbb Z}

\newcommand{\mg}{\mathfrak n }
\newcommand{\mz}{\mathfrak z }

\newcommand{\mh}{\mathfrak h }
\newcommand{\ma}{\mathfrak a }
\newcommand{\mgg}{\mathfrak n }

 \theoremstyle{plain}
 \newtheorem{teo}{Theorem}[section]
 \newtheorem{pro}[teo]{Proposition}
 \newtheorem{corollary}[teo]{Corollary}
 \newtheorem{lem}[teo]{Lemma}

 \renewenvironment{proof}{\noindent \emph{Proof.}}{\hfill $\blacksquare$}
 
 \theoremstyle{remark}
 \newtheorem*{remark}{Remark}

 \newtheorem{example}[teo]{Example}

\newcommand{\lra}{\longrightarrow}

\begin{document}

\title[On a spectral sequence for the cohomology of a nilpotent Lie algebra]
{On a spectral sequence for the cohomology of a nilpotent Lie algebra}

\author{Viviana del Barco}
\address{Viviana del Barco. Facultad de Ciencias Exactas Ingenier\'ia y Agrimensura, Universidad Nacional de Rosario, Av. Pellegrini 250, (2000) Rosario, Argentina.}
\email{delbarc@fceia.unr.edu.ar}

\thanks{Partially supported by CONICET, SCyT-UNR and Secyt-UNC}






\commby{}


\begin{abstract} Given a nilpotent Lie algebra $\mathfrak n$ we construct a spectral sequence which is derived from a filtration of its Chevalley-Eilenberg differential complex and converges to the Lie algebra cohomology of $\mathfrak n$. The limit of this spectral sequence gives a grading for the Lie algebra cohomology, except for the cohomology groups of degree $0$, $1$, $\dim \mg-1$ and $\dim\mg$ as we shall prove. We describe the spectral sequence associated to a nilpotent Lie algebra which is a direct sum of two ideals, one of them of dimension one, in terms of the spectral sequence of the co-dimension one ideal. Also, we compute the spectral sequence corresponding to each real nilpotent Lie algebra of dimension less than or equal to six. 
\medskip


\end{abstract}

\maketitle
\section{Introduction}
In the present work we associate a cohomological spectral sequence to each nilpotent Lie algebra $\mg$. This association is achieved by considering a natural filtration of the Chevalley-Eilenberg differential complex of $\mg$, namely that one defined by the annihilator spaces of the central descending series. The spectral sequence $\{E_r^{p,q}\}$ induced by this filtration converges to the Lie algebra cohomology with trivial coefficients $H^*(\mg)$, yielding the formula 
\[ H^i(\mgg)\simeq \bigoplus _{p+q=i} E_\infty^{p,q}.\qquad (*)\]

Thus one obtains a refinement of the Lie algebra cohomology of $\mg$ by using a na\-tural filtration on the space of 1-forms. The dimensions of the limit terms of the spectral sequence are invariants of $\mg$ from which its Betti numbers can be recovered. In this work we investigate this spectral sequence in the aim to contribute on a further study of the cohomology of nilpotent Lie algebras. Also, an application we pursue is to establish how this spectral sequence constraints the existence of certain geometric structures on $\mg$.

It is known that, contrary to the semisimple case, the explicit calculation of the cohomology of a nilpotent Lie algebra is a hard task in general. Nevertheless a few  particular cases are well understood, such as for Heisenberg Lie algebras \cite{SA}, free 2-step nilpotent Lie algebras \cite{GR-KI-TI,SI} and small dimensional Lie algebras, among others. Also, the cohomology of the nilradicals of parabolic subalgebras in semisimple Lie algebras has been described by Kostant \cite{KO}.

Recall that nilpotent Lie algebras are classified up to dimension seven (see \cite{CI-GR-SH,GO-KH,MA} and references therein). An outstanding conjecture involving the cohomology of nilpotent Lie algebras is the Toral Rank Conjecture \cite{CA-JE,DE-SI,PO,PO-TI2,TI}, which is still open in the general case. It was posed by S. Halperin in 1968 \cite{HA} and states that the total dimension of the Lie algebra cohomology group of $\mgg$ is greater than the total dimension of the Lie algebra cohomology group of the center of $\mgg$. 
Introducing new tools to study the cohomology of nilpotent Lie algebras would feasibly lead to future advances on these problems.

The construction of the spectral sequence under consideration was proposed by S. Salamon in his private correspondence with I. Dotti. In this paper we introduce its formal definition and we 
investigate properties of this spectral sequence in relation to the structure of the nilpotent Lie algebra.

We include here the computation of the spectral sequence of each real nilpotent Lie algebra of dimension $\leq$ 6. Despite the fact that their real cohomology is well known, we find this a necessary step to get conclusions on the properties of $\mg$ that this spectral sequence might reveal. One sees that there are nilpotent Lie algebras having the same Betti numbers but different gradings in (*), showing that the grading given by the spectral sequence is, in fact, more refined that the usual cohomology degree. Also we were able to observe that non-symplectic Lie algebras have zero term $E_\infty^{0,2}$. This yields to a general obstruction to the existence of symplectic structures for  nilmanifolds, a result which is fully developed by the author in \cite{dB2}.

Throughout this work elementary knowledge of spectral sequences will be assumed. 

\medskip

 {\bf Acknowledgments.} This work is part of my Ph.D. Thesis written at FCEIA, Universidad Nacional de Rosario, Argentina and under the supervision of Isabel Dotti. I am very grateful to Isabel Dotti for all her dedication to the guidance work. I am also thankful to Simon Salamon for suggestions which helped me improve the results presented here.

\section{Definition of the spectral sequence}

Let $\mgg$ denote a real Lie algebra. The central descending series of $\mgg$, $\{\mgg^i\}$ is given by
\[\mgg^0=\mgg,\quad  \mgg^i=[\mgg,\mgg^{i-1}],\;\; i\geq 1.
\]
A Lie algebra $\mg$ is $k$-step nilpotent if $\mgg^k=0$ and $\mgg^{k-1}\neq 0$; this number $k$ is called the {\em nilpotency index} of $\mg$. For example, abelian Lie algebras are 1-step nilpotent. For a $k$-step nilpotent Lie algebra it holds $\mgg^{k-1}\subseteq \mz(\mg)$, where $\mz(\mg)$ denotes the center of $\mg$.

\smallskip

The Chevalley-Eilenberg complex of $\mgg$ is
\begin{equation}\label{eq:complex}\R \stackrel{d=0}{\lra} \mgg^*=C^1(\mgg)\stackrel{d}{\lra}C^2(\mgg)\stackrel{d}{\lra}\ldots \stackrel{d}{\lra}C^p(\mgg)\stackrel{d}{\lra}\ldots\quad\end{equation}
where $C^p(\mgg)$ denotes the vector space of skew-symmetric $p$-linear forms on $\mgg$ identified with the exterior product $\Lambda^p\mgg^*$, and the differential $d:\Lambda^{p}\mgg^*\lra \Lambda^{p+1}\mgg^*$ is defined by:
\[ d_p x\,(u_1,\ldots,u_{p+1})=\sum_{1\leq i<j\leq p+1}(-1)^{i+j-1}x([u_i,u_j],u_1,\ldots,\hat{u_i},\ldots,\hat{u_j},\ldots,u_{p+1}). \]

Notice that the differential $d:\mgg\lra \Lambda^2\mgg^*$ coincides with the dual mapping of the Lie bracket $[\,,\,]:\Lambda^2\mgg\lra \mgg$ and for $x,y\in\Lambda^*(\mgg^*)$ we have
\[d(x\wedge y)=dx\wedge y +(-1)^{deg\,x}x\wedge dy. \]

The cohomology of $(C^*(\mgg),d)$ is called the {\em Lie algebra cohomology (with trivial coefficients)} of $\mgg$ and it is denoted by $H^*(\mgg)$.

\smallskip
The following subspaces of the dual of a nilpotent Lie algebra were considered in \cite{SA1} and they lead us to a filtration of the complex in (\ref{eq:complex}). Set
\begin{equation} \label{eq:eq53} 
V_0=0 \quad \mbox{ and } \quad V_i=\{ x \in \mgg^*: d x \in \Lambda^2V_{i-1}\}\qquad i\geq 1.
\end{equation}Notice that $V_1$ is the space of closed 1-forms and that $V_0 \subseteq V_1\subseteq \cdots \subseteq V_i\subseteq \cdots\subseteq \mg^* $. These spaces are dual to those in the central descending series \cite{SA1}. That is, for each $i\geq 0$, $V_i=\{x\in \mgg^*\,:\,x(u)=0,\,\,\forall \,u\in \mgg^i\}=(\mgg^i)^\circ$. In particular, $\mgg$ is $k$-step nilpotent if and only if $V_k=\mgg^*$ and $V_{k-1}\neq \mgg^*$; we define $V_i=0$ for $i<0$.

\medskip

Let $\mgg$ be a $k$-step nilpotent Lie algebra of dimension $m$. The sequence in (\ref{eq:eq53}) defines a filtration in the space of skew symmetric $q$-forms, $\Lambda^q\mg^*$, for all $q=0,\ldots,m$,
\begin{equation}\label{l4}
0=\Lambda^q V_0 \subsetneq \Lambda^q V_1 \subsetneq \ldots \subsetneq \Lambda^q V_{k-1} \subsetneq \Lambda^q V_k=\Lambda^q \mg^*.
\end{equation}
Also $d(\Lambda^qV_i)\subset \Lambda^{q+1} V_i$ so each of these subspaces is invariant under the Lie algebra differential. Thus for each $p$
 \begin{equation}\label{eq7}F^pC^*: \quad 0 \lra \R \lra V_{k-p} \lra \Lambda^2 V_{k-p} \lra \cdots \lra \Lambda^m V_{k-p} \lra 0 \end{equation}
  is a subcomplex of the Chevalley-Eilenberg complex. 
  Observe that $F^p C=0$ (the zero complex) if $p> k$ and $F^p C=C^*$ if $p<0$, hence
  \[\ldots\subseteq 0 \subseteq F^{k-1}C^*\subseteq \ldots \subseteq F^{p+1}C^*\subseteq F^{p}C^*\subseteq \ldots \subseteq F^{1}C^*\subseteq C^*\subseteq \ldots\quad.\]
   Therefore $\{F^pC^*\}$ constitutes a bounded filtration of the complex (\ref{eq:complex}) which gives rise to a spectral sequence $\{E_r ^{p,q}\}
   $  where for $p,q \,\in\Z$ it holds 
\begin{equation}\label{eq:E}
 E_0^{p,q} = \frac{F^pC^{p+q}}{F^{p+1}C^{p+q}}, 
\end{equation}
and $d_0^{p,q}:E_0^{p,q}\lra E_0^{p,q+1}$ is the quotient map induced by the Lie algebra differential. For further details on spectral sequences we refer the reader to \cite[Sect. 5.4]{WE}.
 
 \smallskip
For any nilpotent Lie algebra $\mgg$ the filtration and the spectral sequence that arise in this natural way are bounded because the dimension of $\mgg$ and its nilpotency index are finite. It is well known that given a cohomological complex and a filtration of it, the  associated spectral sequence converges to the cohomology of the original complex. In this case the spectral sequence in (\ref{eq:E}) converges to the Lie algebra cohomology of $\mgg$. This fact is denoted by \[E_r^{p,q}\Rightarrow H^*(\mgg).\]

An isomorphism between two nilpotent Lie algebras preserves the filtrations and induces a spectral sequence isomorphism. This implies that there is a spectral sequence associated to each isomorphism class of nilpotent Lie algebras.
 
\medskip
The zero page of the spectral sequence defined here satisfies
	\begin{equation} \label{l6}
E_0^{p,q} = \frac{\Lambda^{p+q}{V_{k-p}}}{\Lambda^{p+q}{V_{k-(p+1)}}} \qquad \mbox{if}\quad0\leq p\leq k-1, \end{equation} while $E_0^{p,q}=0$ if $p<0$ or $p\geq k$. Thus the (possibly) non-zero elements of total degree $n$ in the $r$-th page are $E^{0,n}_r,$ $E^{1,n-1}_r,\ldots,$ $E_r^{k-1,n-k+1}$ where $k$ is the nilpotency index of $\mgg$.

Each term of the spectral sequence can be computed directly by calculating the Lie algebra differential $d$ paying attention to the filtration. Explicitly,
\begin{equation}\label{eq5}E_r^{p,q}\simeq\frac{A_r^{p,q}}{d(A_{r-1}^{p-r+1,\,q+r-2})+A_{r-1}^{p+1,\,q-1}}\end{equation}
where $A_r^{p,q}=\{x \in \Lambda^{p+q}V_{k-p}:dx \in \Lambda^{p+q+1}V_{k-p-r}\}$. The limit terms of the spectral sequence are
\begin{equation}\label{eq6}E_\infty^{p,q}\simeq\frac{\{x \in \Lambda^{p+q}V_{k-p}:dx =0\}}{d(\{x\in\Lambda^{p+q-1}\mgg^*:dx\in \Lambda^{p+q}V_{k-p}\})+\{x\in\Lambda^{p+q}V_{k-p-1}:dx=0\}}. \end{equation}

\medskip
In the next example we deal with a family of filiform Lie algebras. For each Lie algebra of the family we compute the limit terms of the spectral sequence associated to it with total degree 0,1 and 2; for total degree 2 this is done using previous results on their cohomology. 
\begin{example}\label{pro:pro2} 
Let $m\in\N$ and denote $\mm_0(m)$ the real Lie algebra having a basis of 1-forms $\{e^1,\ldots,e^m\}$ with differential $d:\mg^*\lra \Lambda^2\mg^*$ determined by the following formula \begin{equation} de^i=\left\{
\begin{array}{cl}
0 & \mbox{if} \;\;i=1,2, \\
e^1\wedge e^{i-1} & \mbox{if}\;\; i=3,\ldots,m.
\end{array}
\right. \label{eq25}
\end{equation} 
The Lie algebra $\mm_0(m)$ is filiform, that is, its nilpotency index $k$ is $m-1$, and its cohomology is studied by D. Millionschikov in \cite{MI}. 

Let $\{E_r^{p,q}\}$ be the spectral sequence associated to $\mm_0(m)$. The filtration of the dual of $\mm_0(m)$ is $V_i=span\{e^1,\ldots,e^{i+1}\}$ for $i=1,\ldots,m-1$. Using formula in (\ref{eq6}) it is easy to see that $E_\infty^{p,-p}\simeq \R$ if $p=k-1=m-2$ and $E_\infty^{p,-p}= 0$ otherwise. Also, there is only one non-zero term of total degree 1 in the limit of $\{E_r^{p,q}\}$. Indeed, 
\[E_\infty^{p,1-p}\simeq\frac{\{x \in V_{k-p}:dx =0\}}{\{x\in V_{k-p-1}:dx=0\}}, \] 
which is zero if $k-p\geq 2$ while $E_\infty^{k-1,2-k}\simeq V_1$. For any nilpotent Lie algebra the limit of the spectral sequence associated to it has only one non-zero element of total degree 0 and only one of total degree 1 as we shall see in the next section.

According to \cite[Lemma 5.2]{MI}, the second cohomology group $H^2(\mm_0(m))$ admits a basis composed by the cohomology class of $e^1\wedge e^m$ and those of $\omega_s$, where
\[\omega_s=\frac12 \sum_{i=2}^{2s-1}(-1)^i\,e^i\wedge e^{2s+1-i},\;\; s=2,\ldots, 
[(m+1)/2]; \] the brackets denote the integer part.
Each $\omega_s$ is closed and non-exact. 

This basis of $H^2(\mm_0(m))$ can be used to compute the degree 2 terms of the limit of spectral sequence $\{E_r^{p,q}\}$; namely, $E^{0,2}_\infty,$ $E^{1,1}_\infty,\ldots,E_\infty^{m-2,4-m}$. Recall from Eq. (\ref{eq6}) 
\[E_\infty^{p,2-p}\simeq\frac{\{x \in \Lambda^{2}V_{m-1-p}:dx =0\}}{d(V_{m-p})+\{x\in\Lambda^{2}V_{m-p-2}:dx=0\}}\;\mbox{ for } p=0,\ldots,m-2.\]
Since $V_i=span\{e^1,\ldots,e^{i+1}\}$ for $i=1,\ldots,m-1$, it holds
\[e^1\wedge e^m \in\Lambda^2V_{m-1},\; \mbox{ while }\;\omega_s\in \Lambda^2 V_{2s-2}, \quad s=2,\ldots,[(m+1)/2];\]
also $\omega_s\notin \Lambda^2 V_{2s-3}$. Hence $e^1\wedge e^{m}$ defines a non-trivial element in $E_\infty^{0,2}$. In addition, $\omega_s$ defines a non-zero element in $E_\infty^{p,2-p}$ if and only if $p=m-2s-1$. This gives different situations depending on whether $m$ is even or odd. In fact
\begin{equation}\label{eq20}
\dim E_\infty^{0,2}= \left\{\begin{array}{cl}
1 & \mbox{if }m \mbox{ is even},\\
2 & \mbox{if }m \mbox{ is odd},\end{array}\right. \;\mbox{and}\;
\dim E_\infty^{p,2-p}=\left\{\begin{array}{cl}
0 & \mbox{if } p\equiv m \,(\mbox{mod }2)\\
1 & \mbox{if } p\not\equiv m \,(\mbox{mod }2)\end{array}\right. \end{equation}
for $p=1,\ldots,m-2$.
\end{example}

\section{General properties}
In the previous section, to each nilpotent Lie algebra $\mgg$ was associated a natural filtration of the Chevalley-Eilenberg differential complex. This filtration gives rise to a spectral sequence that converges to the Lie algebra cohomology. Thus each cohomology group $H^i(\mgg)$ can be written as a direct sum of the limit terms of the spectral sequence. Namely
\begin{equation}
\label{eq8}H^i(\mgg)\simeq \bigoplus _{p+q=i} E_\infty^{p,q} \quad \text{ for all } i=0,\ldots, \dim \mg.
\end{equation} 
Thus the spaces $E_\infty^{p,q}$ refine the Lie algebra cohomology of $\mg$ and, in view of Nomizu's theorem \cite{NO}, the de Rham cohomology of a nilmanifold $\Gamma\backslash N$ where $\mg$ is the Lie algebra of $N$. Because of this essential property, we might call them the {\em intermediate cohomology groups} of $\mg$ (or $N$). \medskip

Computing all terms of the spectral sequence associated to a nilpotent Lie algebra $\mgg$ using the definition or Eqns. (\ref{eq5}), (\ref{eq6}) might take a large amount of work. For this reason we introduce properties that intend to facilitate this work. First, we review results on closed and exact $m-1$-forms of an $m$ dimensional nilpotent Lie algebra.

\begin{lem}\label{lm1}
Let $\mg$ be an $m$ dimensional $(r+1)$-step nilpotent Lie algebra with differential $d:\mg^*\lra \Lambda^2\mg^*$. Let $n_0$ be the dimension of $V_1=\ker d$ and denote $\beta_1,\,\beta_2,\ldots,\,\beta_{n_0}$ a basis of $V_1$. Then every $\sigma \in \Lambda^{m-1}\mg^*$ is closed. Moreover $\sigma$ is exact if and only if it is divisible by $\beta_1\wedge \beta_2\wedge\cdots\wedge\beta_{n_0}$.
\end{lem}

\begin{remark} This lemma joins a result due to Benson and Gordon in \cite{BE-GO} and its converse proved by Yamada in \cite{YA}. In both works, the authors add the hypothesis that $\mg$ is a symplectic Lie algebra. Nevertheless we noticed that this is valid for any nilpotent Lie algebra. We include the proof here to make this fact explicit and for completeness of the exposition.\end{remark}

\begin{lmproof} {\em \ref{lm1}}.
Let $\ma^i$ denote a vector space complementary to $\mg^{i+1}$ in $\mg^i$ where $\mg^i$ is the $i$-th term of the central descending series of $\mg$. For $i=0,1,\ldots,r$
\[\mg^{i}=\mg^{i+1}+\;\ma^{i};\] define $n_i=\dim \,\ma^{i}$. Denote
 \[\Lambda^{i_0,i_1,\ldots,i_{r}}=\Lambda^{i_0}(\ma^0)^*\wedge \Lambda^{i_1}(\ma^1)^*\wedge \cdots \wedge \Lambda^{i_{r}}(\ma^{r})^*\,\subseteq\,\Lambda^{i_0+i_1+\ldots+i_{r}}\mg^*.\]
Then for $s=0,\ldots,m$
\[ \Lambda^s\mg^*=\sum_{i_0+i_1+\ldots+i_{r}=s}\Lambda^{i_0,i_1,\ldots,i_{r}}.\]
Notice that it is possible to choose $\ma^0$ such that $(\ma^0)^*=V_1$. Let $\beta_1,\,\beta_2,\ldots,\,\beta_{n_0}$ be a basis of $(\ma^0)^*$.

Assume $\eta\in\Lambda^{i_0,\ldots,i_{r}}$ then each term of $d\eta$ belongs to a subspace $\Lambda^{t_0,t_1,\ldots,t_r}$ such that there exists an index $j\geq 1$ with $t_j=i_j-1$ and $t_0+t_1+\ldots+t_{j-1}=i_0+i_1+\ldots+i_{j-1}+2$.

Suppose in particular that $\eta\in\Lambda^{m-1}\mg^*$, this implies $i_0+i_1+\ldots+i_{r}=m-1$. Then for any $j\geq 1$ we have $i_0+i_1+\cdots +i_{j-1}\geq n_0+n_1+\cdots +n_{j-1}-1$. In fact if for some $j$ holds $i_0+i_1+\cdots +i_{j-1}< n_0+n_1+\cdots +n_{j-1}-1$, then $\sum_{l=1}^r i_l<\sum_{l=1}^r n_l-1=2m-1$ leading to a contradiction.
Fixed the term of $d\eta$ in $\Lambda^{t_0,t_1,\ldots,t_r}$, let $j$ be the index specified in the previous paragraph. It verifies $t_0+t_1+\ldots+t_{j-1}=i_0+i_1+\ldots+i_{j-1}+2>n_0+n_1+\cdots+n_{j-1}$ implying $d\eta=0$. Therefore any $m-1$-form is closed.

Let $\sigma=d\alpha$ be an $m-1$-exact form. Clearly each component $\eta$  of $\alpha$ belongs to some $\Lambda^{i_0,\ldots,i_{r}}$ with $i_0+\ldots+i_{r}=m-2$. As before, the term of $d\eta$ belonging to $\Lambda^{t_0,t_1,\ldots,t_r}$ satisfies $t_0+ t_1+\cdots+t_j=n_0+n_1+\cdots+n_j$ for all $j\geq 1$. In particular, $t_0=n_0$ which implies that $\sigma$ is divisible by $\beta_0\wedge\cdots  \wedge \beta_{n_0}$.

To prove that any $m-1$-form divisible by $\beta_0\wedge\cdots  \wedge \beta_{n_0}$ is exact, a dimensional argument is used. Denote $B^{m-1}$ the set of exact $m-1$-forms, then $B^{m-1}$ is a subset of $ \sum_{n_0+i_1+\cdots+i_r=m-1} \Lambda^{n_0,i_1,\ldots,i_r}$. From Poincar\'{e} duality $\dim H^{m-1}(\mg)=\dim H^1(\mg)=n_0$.
Moreover, the result above states $\dim H^{m-1}(\mg)=\dim \Lambda^{m-1}\mg^*-\dim B^{m-1}$. Combining these two formulas one leads to $\dim B^{m-1}=n_1+\cdots+n_r$ and therefore \linebreak $ B^{m-1}=\sum_{n_0+i_1+\cdots+i_r=m-1} \Lambda^{n_0,i_1,\ldots,i_r}$.
\end{lmproof}

\medskip

The next result states that the limit terms of the spectral sequence of total degree $0,\,1,\,\dim \mg-1$ and $\dim\mg$ are mostly zero.
\begin{teo}\label{pro:pro3} Let $\mg$ be an $m$ dimensional $k$-step nilpotent Lie algebra with diffe\-rential $d:\mg^*\lra\Lambda^2\mg^*$ and let $\{E_r^{p,q}\}$ be the spectral sequence associated to $\mg$. Then
\begin{enumerate}
\item  $E_\infty^{k-1,1-k}=H^0(\mg)=\R$ and hence $E_\infty^{p,-p}(\mg)=0 $ for all $p=0,\ldots, k-2$.
\item  $E_\infty^{k-1,2-k}=H^1(\mg)=\ker d$ and hence $E_\infty^{p,1-p}(\mg)=0 $ for all $p=0,\ldots, k-2$.
\item $E_\infty^{0,m-1}=H^{m-1}(\mg)$ and hence $E_\infty^{p,m-1-p}=0$ for all $p=1,\ldots ,k-1$.
\item $E_\infty^{0,m}\simeq H^m(\mg)\simeq\R$ and hence $E_\infty^{p,m-p}=0$ for all $p=,1\ldots, k-1$.
\end{enumerate}
\end{teo}

\begin{proof} The first two assertions follow straight from the definition of the spectral sequence of $\mg$. Moreover, it is well known that the top cohomology group $H^m(\mgg)$ of a nilpotent Lie algebra $\mgg$ is spanned by a class of the form $e^1\wedge \ldots\wedge e^m$ if $\{e^1,\ldots,e^m\}$ is a basis of $\mgg^*$ \cite{KOZ}. This implies $E_\infty^{0,m}\simeq H^m(\mg)\simeq\R$ which combined with Eq. (\ref{eq8}) proves (4).
 
To prove (3), we make use of the lemma above and Eq. (\ref{eq6}) which together imply \[E_\infty^{0,m-1}=
\frac{\Lambda^{m-1}\mg^*}{B^{m-1}+\Lambda^{m-1}V_{k-1}}.\]
It is sufficient to show that $\Lambda^{m-1}V_{k-1}\subseteq B^{m-1}$ since it implies \[E_\infty^{0,m-1}=\frac{\Lambda^{m-1}\mg^*}{B^{m-1}}=H^{m-1}(\mg).\] 

Notice that if $\dim V_{k-1}<m-1$ then $\Lambda^{m-1}V_{k-1}=0$ and it is clearly contained in $B^{m-1}$. When $\dim V_{k-1}=m-1$, $\dim \ma^{k-1}=1$ (in the notation of Lemma \ref{lm1}) and the subspace $\Lambda^{m-1}V_{k-1}$ coincides with $\Lambda^{n_0,n_1,\ldots,n_{k-2},0}$ also a subset of $B^{m-1}$ because of the lemma above. Thus $E_\infty^{0,m-1}=H^{m-1}(\mg)$ in any case.
Using Eq. (\ref{eq8})
one concludes that
$E_\infty^{p,m-1-p}=0$ for all $1\leq p\leq k-1$.
\end{proof}

\medskip

In Lie algebra cohomology the well known K\"unneth formula relates the Lie algebra cohomology of a Lie algebra which is a direct sum of ideals, with the cohomology of its summands. 

When the nilpotent Lie algebra $\mgg$ can be decomposed as a direct sum of two ideals being one of them of dimension one, a similar formula relates the spectral sequence of $\mgg$ with those of its summands.

\begin{teo}\label{teo:teo24} Let $\mg$ be a $k$-step nilpotent Lie algebra which can be decomposed as a direct sum of ideals $\mg=\R\oplus\mh$. Denote by $\{E_r^{p,q}\}$ and $\{\tilde{E}_r^{p,q}\}$ the spectral sequences associated to $\mg$ and to $\mh$ respectively. Then it holds
\begin{enumerate}
\item $E_r^{p,q}=0$ if $p+q<0$,
\item $E_r^{k-1,1-k}\simeq\R$ and $E_r^{p,-p}=0$ for all $p\neq k-1$. 
\item $E_r^{k-1,2-k}\simeq \tilde{E}_r^{k-1,2-k}\oplus \R$,
\item $E_r^{p,1-p}\simeq \tilde{E}_r^{p,1-p}$ if $0\leq p \leq k-2$ and $E_r^{p,1-p}=0$ if $p<0$ or $p\geq k$.
\item $E_r^{p,q}\simeq \tilde{E}_r^{p,q}\oplus \tilde{E}_r^{p,q-1}$ if $p+q\geq 2$.
\end{enumerate}
\end{teo}
 
 \begin{skproof} Notice that under the hypothesis of the theorem $\mh$ is $k$-step nilpotent. We will only show some equalities that lead to the complete proof.

The first two assertions follow from the definition of the spectral sequence and Eq. (\ref{l6}). Suppose $\mg=\R x\oplus \mh$ and denote by $x^*$ the element in $\mg^*$ such that $x^*(x)=1$, $x^*(\mh)=0$ and identify  $\mh^*$ with a subset of $\mg^*$. Clearly $dx^*=0$ and $\mh^*$ is invariant by $d$. Also the restriction of $d$ to $\mh^*$ is the differential of the Lie algebra $\mh$.

The subsets  $V_0\subseteq V_1\subseteq\cdots \subseteq V_k=\mg^*$ and $\widetilde{V}_0\subseteq \widetilde{V}_1\subseteq \cdots\subseteq \widetilde{V}_k=\mh^*$ which filter $\mg^*$ and  $ \mh^*$ respectively satisfy:
\begin{equation}\label{rel}\widetilde{V}_0=V_0=0,\qquad V_i=\widetilde{V}_i\oplus \R x^*,\;i=1,\ldots,k .\end{equation}

Eq. (\ref{l6}) gives the initial term of $\{E_r^{p,q}\}$: for total degree one, $E_0^{p,1-p}=V_{k-p}/V_{k-p-1}$ which by (\ref{rel}) is
\[\begin{array}{lcl}
E_0^{k-1,2-k}&=&\wtV_1\oplus \R x^* \text{ and }\vspace{0.2cm}\\
E_0^{p,1-p}&\simeq& \wtV_{k-p}/\wtV_{k-p-1}=\tilde{E}_0^{p,1-p} \qquad\text{ if  }p\neq k-1. \end{array}\]

If $p+q\geq 2$ the equality $\Lambda^{p+q}\,V_{k-p}=\Lambda^{p+q}\, \wtV_{k-p}\oplus (\R x^*\wedge \Lambda^{p+q-1}\,\wtV_{k-p})$ used in Eq. (\ref{l6}) yields to
\begin{eqnarray}
E_0^{p,q}&=&
\frac{\Lambda^{p,q} \wtV_{k-p}\oplus ( \R x^*\wedge \Lambda^{p+q-1}\wtV_{k-p})}{\Lambda^{p,q} \wtV_{k-p-1}\oplus (\R x^*\wedge \Lambda^{p+q-1}\wtV_{k-p-1})}\simeq \frac{\Lambda^{p,q} \wtV_{k-p}}{\Lambda^{p,q} \wtV_{k-p-1}}\oplus \frac{ \Lambda^{p+q-1}\wtV_{k-p}}{\Lambda^{p+q-1}\wtV_{k-p-1}}\nonumber \\
\nonumber\\
&\simeq & \tilde{E}_0^{p,q}\oplus \tilde{E}_0^{p,q-1}.\nonumber
\end{eqnarray}

When $r\geq 1$ the key is to use Eq. (\ref{eq5}). Suppose $p=k-1$, then
\begin{equation}\label{eq32}E_r^{k-1,q}=\frac{\{y\in \Lambda^{k+q-1}V_1:dy=0\}}{d(\{y \in \Lambda^{k+q-2}V_r:dy\in \Lambda^{k+q-1}V_1\})}.\end{equation}
If $q=2-k$ then $E_r^{k-1,2-k}=\wtV_1\oplus \R x^*\simeq \tilde{E}_r^{k-1,2-k}\oplus\R$; if $q\geq 3-k$ (or equi\-valently $k+q-1\geq 2$) 
\[ \Lambda^{k+q-1}V_1=\Lambda^{k+q-1}\wtV_1\oplus (\R x^*\wedge \Lambda^{k+q-2}\wtV_1),\]
therefore $\omega \in  \Lambda^{k+q-1}V_1$ if and only if $\omega=\omega_1+x^*\wedge \omega_2$ with $\omega_1\in\Lambda^{k+q-1}\wtV_1$ and $\omega_2\in \Lambda^{k+q-2}\wtV_1$.
Moreover $\omega$ is closed if and only if $d\omega_1+x^*\wedge d\omega_2=0$, condition only satisfied in the case that $ \omega_1$ and $\omega_2$ are simultaneously closed. So it is possible to rewrite the numerator in (\ref{eq32}) as
\[\{x\in \Lambda^{k+q-1}\wtV_1:dy=0\}\oplus (\R x^*\wedge \{x\in \Lambda^{k+q-2}\wtV_1:dy=0\}).\]

To describe the denominator of the same equation it is necessary to consider two cases:  $k+q-2=1$ (or equivalently $q=3-k$) and $k+q-2\geq 2$. In the first situation 
\[d(\{y \in \Lambda^{k+q-2}V_r:dy\in \Lambda^{k+q-1}V_1\})= d(\{y \in V_r:dy\in \Lambda^{2}V_1\})=d(V_2)=d(\wtV_2),\] reaching
\begin{eqnarray}E_r^{k-1,3-k}&=&\frac{\{x\in \Lambda^{2}\wtV_1:dy=0\}\oplus (\R x^*\wedge \{x\in \wtV_1:dy=0\})}{d(\wtV_2)}\nonumber  \\
\nonumber\\
&\simeq&  \tilde{E}_r^{k-1,3-k}\oplus (\underbrace{\R x^*\wedge \wtV_1}_{\simeq \wtV_1})\simeq \tilde{E}_r^{k-1,3-k}\oplus \tilde{E}_r^{k-1,2-k}.\nonumber \end{eqnarray}

In the case $k+q-2\geq 2$ (or equivalently $q\geq 4-k$) one has
 \begin{eqnarray}
 \Lambda^{k+q-2} V_r&=&\Lambda^{k+q-2}\wtV_r\oplus \,(\R x^*\,\wedge\, \Lambda^{k+q-3}\wtV_r),\nonumber\\
 \Lambda^{k+q-1}V_1&=&\Lambda^{k+q-1}\wtV_1\oplus (\R x^* \wedge \Lambda^{k+q-2}\wtV_1).\nonumber\end{eqnarray}
 Thus any element $\omega$ of $\Lambda^{k+q-2} V_r$ can be written as $\omega=\omega_1+x^*\wedge \omega_2$ where\linebreak
$\omega_1\in\Lambda^{k+q-2}\wtV_r$ and $\omega_2\in \Lambda^{k+q-3}\wtV_r$. In addition $d\omega=d\omega_1+x^* \wedge d\omega_2$ is an ele\-ment into $\Lambda^{k+q-1}V_1$  if and only if $
d\omega_1\in \Lambda^{k+q-1}\wtV_1 $ and $d\omega_2\in \Lambda^{k+q-2}\wtV_1.$ All this imply
\[ d(\{y \in \Lambda^{k+q-2}V_r:dy\in \Lambda^{k+q-1}V_1\})=U\oplus W\]
where \vspace{-0.4cm}
\begin{eqnarray}
U&=&d(\{y \in \Lambda^{k+q-2}\wtV_r:dy\in \Lambda^{k+q-1}\wtV_1\}),\nonumber \\
W&=&\R x^*\wedge d(\{y \in \Lambda^{k+q-3}\wtV_r:dy\in \Lambda^{k+q-2}\wtV_1\}).\nonumber\end{eqnarray}
 
Combining the formulas obtained for the numerator and denominator in Eq. (\ref{eq32}) it follows
\begin{equation}
E_r^{k-1,q}\simeq \tilde{E}_{r}^{k-1,q}\oplus \tilde{E}_r^{k-1,q-1}.\nonumber\end{equation}

For those $p\neq k-1$ the proof is analogous.
\end{skproof}
\medskip

Applying an inductive procedure, one proves a similar formula for Lie algebras having an abelian direct factor of dimension $s$, $s\geq 1$.

\begin{corollary} Let $\mg$ be a nilpotent Lie algebra that decomposes as a direct sum of ideals $\R^s\oplus \mh$ for some $s\geq 1$. Let $\{E_r^{p,q}\}$ and $\{\tilde{E}_r^{p,q}\}$ be the spectral sequences associated to $\mg$ and $\mh$ respectively. Then each term of the spectral sequence $\{E_r^{p,q}\}$ can be written as a sum of certain terms of the spectral sequence $\{\tilde{E}_r^{p,q}\}$. In particular $\{E_r^{p,q}\}$ degenerates at $r_0$ if and only if $\{\tilde{E}_r^{p,q}\}$ does. \end{corollary}
\smallskip

\begin{example}\label{example} Let $\mgg$ be the four dimensional Lie algebra with basis $\{e_1,e_2,e_3,e_4\}$ and non-zero bracket $[e_2,e_3]=e_4$. This Lie algebra is 2-step nilpotent and it is isomorphic to $\R\oplus \mh_3$; we apply Theorem \ref{teo:teo24} to compute the limit of the spectral sequence of $\mg$, $\{E_r^{p,q}\}$, by means of the spectral sequence associated to $\mh_3$, $\{\tilde{E}_r^{p,q}\}$. 


Recall that, $E_\infty^{p,q}=0$ if $p<0$ or $p\geq 2$. Using the formulas in the previous theorem for the limit of the spectral sequence of $\mg$, we have
\begin{enumerate}
\item $E_\infty^{p,q}=0$ if $p+q<0$,
\item $E_\infty^{0,0}=0$, $\quad E_\infty^{1,-1}=\R$,
\item $E_\infty^{1,0}\simeq \tilde{E}_\infty^{1,0}\oplus \R=span \{[e^2],[e^3]\}\oplus \R$,
\item $E_\infty^{0,1}\simeq \tilde{E}_\infty^{0,1}=0$,
\item 
  \begin{enumerate}
  \item $E_\infty^{0,2}\simeq \tilde{E}_\infty^{0,2}\oplus \tilde{E}_\infty^{0,1}\simeq span\{[e^2\wedge e^4],[e^3\wedge e^4]\}$,
  \item $E_\infty^{1,1}\simeq \tilde{E}_\infty^{1,1}\oplus \tilde{E}_\infty^{1,0}\simeq span\{[e^2],[e^3]\}$,
  \item $E_\infty^{0,3}\simeq \tilde{E}_\infty^{0,3}\oplus \tilde{E}_\infty^{0,2}$
  
   $\phantom{E_\infty^{0,3}}\simeq span\{[e^2\wedge e^3 \wedge e^4]\}\oplus span\{[e^2\wedge e^4],[e^3\wedge e^4]\}$,
  \item $E_\infty^{1,2}\simeq \tilde{E}_\infty^{1,2}\oplus \tilde{E}_\infty^{1,1}=0$ since $\tilde{E}_\infty^{1,1}=\tilde{E}_\infty^{1,2}=0$,
  \item $E_\infty^{0,4}\simeq \tilde{E}_\infty^{0,4}\oplus \tilde{E}_\infty^{0,3}\simeq span\{[e^2\wedge e^3\wedge e^4]\}$.
  \end{enumerate}
\end{enumerate}

It is not hard to verify the isomorphisms established for the non-zero limit terms of the spectral sequence $\{\tilde{E}_r^{p,q}\}$.
\end{example}

\medskip
We conclude this section with some comments on the behavior of the number $r_0$ for which the spectral sequence degenerates, i.e. $r_0$ is the minimum $r$ for which $d_r^{p,q}=0$ and consequently $E_r^{p,q}=E_\infty^{p,q}$ for each $r\geq r_0$ and $p,q\,\in\Z$. This number is related to the nilpotency index and the dimension of the nilpotent Lie algebra $\mg$.

When the Lie algebra $\mg$ is $k$-step nilpotent it is easy to see that for any $r\geq k$, the differential $d_r^{p,q}:E_r^{p,q}\lra E_r^{p+r,q-r+1}$ has the zero space as either its domain or target space (or both). Hence, the minimum $r_0$ is at most $k$. Namely,
\begin{pro}\label{pro:pro1} Let $\mg$ be a $k$-step nilpotent Lie algebra and $\{E_r^{p,q}\}$ the spectral sequence associated to $\mg$. For any $r\geq k$ it holds $E_r^{p,q}=E_\infty^{p,q}$ for all $p,q\,\in\Z$.
\end{pro}

\smallskip

The family of nilpotent Lie algebras in Example \ref{pro:pro2} shows that it is not possible to find a common $r_0$ valid for all nilpotent Lie algebras: for any $m \in\N$ the spectral sequence $\{E_r^{p,q}\}$ associated to the nilpotent Lie algebra $\mm_0(2m)$ does not degenerate in the $m-1$ step.
Indeed, denote with $V_i$, $i=1,\ldots,2m-1$ the filtration of $\mm_0(2m)^*$. Recall that $V_i= \{e^1,\ldots,e^{i+1}\}$, then the 2-form
\[\sigma= 
e^2\wedge e^{2m}-e^3\wedge e^{2m-1}+\cdots +(-1)^{m-2}e^{m}\wedge e^{m+2}\]
is not an element of $\Lambda^2V_{2m-2}$ and by Eq. (\ref{eq25}) its differential is
\[ d\sigma= (-1)^{m-2} \; e^1\wedge e^{m}\wedge e^{m+1}.\]
Thus $d\sigma\in\Lambda^3 V_{m}$ since $V_m=span\{e^1,\ldots,e^{m+1}\}$.
Hence $\sigma$ and $e^1\wedge e^{2m}$ define linearly independent elements in 
\[
E_{m-1}^{0,2}(\mm_0(2m))\simeq \frac{\{x \in \Lambda^{2}\mg^*:dx \in \Lambda^{3}V_{m}\}}{\{x\in\Lambda^{2}V_{2m-2}:dx\in\Lambda^{3}V_{m}\}}
\] (see Eq. (\ref{eq5})) so $\dim E_{m-1}^{0,2}(\mm_0(2m))\geq 2$. Instead, by Eq. (\ref{eq20})\[
\dim \,E_\infty^{0,2}(\mm_0(2m))=1,\] therefore $E_{m-1}^{p,q}$ is not the limit of the spectral sequence.

\begin{remark} It would be interesting to find numbers $r_0(m)$ for which the spectral sequence of any nilpotent Lie algebra $\mg$ of dimension $m$ degenerates at $r\leq r_0(m)$. From the last example, this number $r_0(m)$ must be at least $m/2$. We have examples to believe that $r_0(m)$ would be $\left\lceil m/2\right\rceil$, the least integer not smaller than $m/2$.
\end{remark}

\section{The spectral sequence in low dimensions}

Nilpotent Lie algebras over the real numbers are classified up to dimension seven. Though six is the highest dimension in which there do not exist continuous families \cite{GO-KH,MA}. In this section we present the limit of the spectral sequences associated to nilpotent Lie algebra up to dimension six. The computations were made by hand using, when possible, the properties proved in the previous section. Later, we checked the calculations with a computational program we developed. The information obtained is presented in tables as we explain next. 

\smallskip
Given a $k$-step nilpotent Lie algebra $\mgg$ of dimension $m$ and with spectral sequence $\{E_r^{p,q}\}$ we build up tables to display the dimensions of $E_0^{p,q},\,E_1^{p,q},\,\ldots,E_r^{p,q}=E_\infty^{p,q}$, precisely, $r+1$ tables having $k$ rows and $m+1$ columns. In the first column we put the terms of total degree 0, in the second column the ones of total degree 1, and so on. Set $e_r^{p,q}=\dim E_r^{p,q}$, thus the table corresponding to $r$-th term of the spectral sequence of $\mgg$ is as the one below.
\begin{table}[h]
\centering
		\begin{tabular}{|c|c|c|c|c|}
		\hline 
		$e_r^{k-1,1-k}$&$e_r^{k-1,2-k}$ & $e_r^{k-1,3-k}$ &  $\cdots$ &$ e_r^{k-1,m+1-k}$ \\		
		\hline
		\vdots&\vdots & \vdots &  $\cdots$  &\vdots \\
		\hline 		
$e_r^{1,-1}$&$e_r^{1,0}$ & $e_r^{1,1}$ & $\cdots$  &$ e_r^{1,m-1} $ \\		
		\hline 
		$e_r^{0,0}$ & $e_r^{0,1}$ & $e_r^{0,2}$ & $\cdots$ &$ e_r^{0,m} $ \\
		\hline 
		\end{tabular}
		
		\vspace{0.5cm}
		\label{tab1}
	
\end{table}
\pagebreak

Some of the properties  of the spectral sequence established along the previous section can be read off the diagram corresponding to $E_\infty^{p,q}$. Indeed, because of Theorem \ref{pro:pro3} the first two columns and the last two columns consist of zeros except for the top or bottom numbers. In fact, the first column (resp. last column) has a number 1 at the top (resp. bottom) of the table and the second column (resp. penultimate column) has the $\dim \ker (d:\mg^*\lra \Lambda^2\mg^*)$ at the top (resp. bottom) of the table.

Furthermore a consequence of Eq. (\ref{eq8}) is that the sum of the elements of the $i$-th column gives as result the $i$-th Betti number of the Lie algebra $\mgg$, i.e. \[\beta_i=\dim H^i(\mgg) = \sum_{p+q=i} e_\infty^{p,q}. \]
For nilpotent Lie algebras the Poincar\'e duality holds (see \cite{KOZ}) hence the sum of the elements of the $i$-th column coincides with the sum of the elements in the $(m-i)$-th column. 

Theorem \ref{teo:teo24} gives instructions to construct the table of a Lie algebra of the form $\mg=\R\oplus \mh$ once we know the table of $\mh$: the $i$-th column of the table corresponding to $\mg$ is obtained by adding up the $i$-th with the $(i-1)$-th columns of the table corresponding to $\mh$ with the only exception of the first column.

\smallskip

The notation we use to describe the Lie algebras is that one employed by S. Salamon in \cite{SA1}. That is, a Lie algebra denoted as $\mg=(0,0,12,13,23,14+25)$ is that one admitting a basis of 1-forms $\{e^1,e^2,e^3,e^4,e^5,e^6\}$ where the Lie algebra differential satisfies $ de^1=de^2=0$ and
\[ de^3=e^1\wedge e^2,  \qquad de^4=e^1 \wedge e^3,   \qquad de^5=e^2\wedge e^3, \qquad de^6=e^1\wedge e^4+e^2 \wedge e^5.\] 
The comparison between this notation and that one used by Magnin \cite{MA} and Morozov \cite{MOR} in their classifications is available in \cite{SA1} for dimension six.

If a listed nilpotent Lie algebra $\mg$ admits a decomposition as a direct sum of ideals of the form $\mgg=\R^s \oplus \mh$ we will make this fact clear for the reader to compare the tables of $\mg$ and $\mh$ according to Theorem \ref{teo:teo24}. 
\medskip

{\bf \Large Dimension three}
\smallskip

$\mh_3=(0,0,12)$

{\tiny
$$\begin{array}{cccc}
E_0 & E_1 & E_2=E_\infty \\
\begin{tabular}{|c|c|c|c|}
	\hline
1&	2	&1&	0\\
  \hline
  0&1&2&1  \\
\hline
\end{tabular}
&
\begin{tabular}{|c|c|c|c|}
	\hline
	1&2	&1&	0  \\
  \hline
  0&1&2&1\\
\hline
\end{tabular}
&
\begin{tabular}{|c|c|c|c|}
	\hline
	1&2	&0&	0  \\
  \hline
  0&0&2&1\\
\hline
\end{tabular}
\end{array}$$}

{\bf \Large Dimension four}
\smallskip
\begin{enumerate}

\item\label{dim4-2} $\mgg=(0,0,0,12)$, $\mg\simeq \R\oplus \mh_3$.
{\tiny
$$\begin{array}{cccc}
E_0 & E_1 & E_2=E_\infty \\
\begin{tabular}{|c|c|c|c|c|}
	\hline
1&	3	&3&	1&0\\
	\hline
  0&1&3&3&1\\
\hline
\end{tabular}
&
\begin{tabular}{|c|c|c|c|c|}
	\hline
	1&3	&3&	1&0\\
	\hline
  0&1&3&3&1\\
\hline
\end{tabular}
&
\begin{tabular}{|c|c|c|c|c|}
	\hline
	1&3	&2&	0&0\\
	\hline
  0&0&2&3&1\\
\hline
\end{tabular}
\end{array}$$}

\item\label{dim4-1} $\mgg=(0,0,12,13)$
{\tiny
$$\begin{array}{cccc}
E_0 & E_1 & E_2=E_\infty \\
\begin{tabular}{|c|c|c|c|c|}
	\hline
1&	2	&1&	0&0\\
	\hline
  0&1&2&1&0\\
  \hline
  0&1&3&3&1\\
\hline
\end{tabular}
&
\begin{tabular}{|c|c|c|c|c|}
	\hline
	1&2	&1&	0&0\\
	\hline
  0&1&2&1&0\\
  \hline
  0&1&2&2&1\\
\hline
\end{tabular}
&
\begin{tabular}{|c|c|c|c|c|}
	\hline
	1&2	&0&	0&0\\
	\hline
  0&0&1&0&0\\
  \hline
  0&0&1&2&1\\
  \hline
\end{tabular}
\end{array}$$}

\end{enumerate}

{\bf \Large Dimension five}
\smallskip
 
$(1)\,\bullet \,\mg=(0,0,12,13,14+23)$:\hspace{-0.5cm}
	\begin{flushleft}	
\tiny $$\begin{array}{ccc}
E_0 & E_1 & E_2 \\
\\
\begin{tabular}{|c|c|c|c|c|c|}
	\hline
	1&2	&1&	0&	0&	0\\
	\hline
0&1&	2&	1&	0&	0\\
  \hline
  0&1&	3&	3&	1&	0\\
  \hline
  0&1&	4&	6&	4&	1\\
  \hline
	\end{tabular}
&

\begin{tabular}{|c|c|c|c|c|c|}
\hline
	1&2	&1&	0&	0&	0\\
	\hline
0&1&	2&	1&	0&	0\\
  \hline
  0&1&	2&	2&	1&	0\\
  \hline
  0&1&	2&	2&	2&	1\\
  \hline	\end{tabular}
&
\begin{tabular}{|c|c|c|c|c|c|}
\hline
	1&2	&0&	0&	0&	0\\
	\hline
0&0&	1&	0&	0&	0\\
  \hline
  0&0&	0&	2&	0&	0\\
  \hline
  0&0&	2&	1&	2&	1\\
  \hline	\end{tabular}
\end{array} $$
\end{flushleft}

$(2)\,\bullet \,\mg=(0,0,12,13,14)$:
	\begin{flushleft}	
\tiny $$\begin{array}{ccc}
E_0 & E_1 & E_2 \\
\\
\begin{tabular}{|c|c|c|c|c|c|}
	\hline
	1&2	&1&	0&	0&	0\\
	\hline
0&1&	2&	1&	0&	0\\
  \hline
  0&1&	3&	3&	1&	0\\
  \hline
  0&1&	4&	6&	4&	1\\
  \hline
	\end{tabular}
&

\begin{tabular}{|c|c|c|c|c|c|}
\hline
	1&2	&1&	0&	0&	0\\
	\hline
0&1&	2&	1&	0&	0\\
  \hline
  0&1&	2&	2&	1&	0\\
  \hline
  0&1&	2&	2&	2&	1\\
  \hline	\end{tabular}
&
\begin{tabular}{|c|c|c|c|c|c|}
\hline
	1&2	&0&	0&	0&	0\\
	\hline
0&0&	1&	0&	0&	0\\
  \hline
  0&0&	0&	2&	0&	0\\
  \hline
  0&0&	2&	1&	2&	1\\
  \hline	\end{tabular}
\end{array} $$
\end{flushleft}

$(3)\,\bullet \,\mg=(0,0,12,13,23)$:
	\begin{flushleft}	
\tiny $$\begin{array}{ccc}
E_0 & E_1 & E_2 \\
\\
\begin{tabular}{|c|c|c|c|c|c|}
	\hline
	1&2	&1&	0&	0&	0\\
	\hline
0&1&	2&	1&	0&	0\\
  \hline
  0&2&	7&	9&	5&	1\\
  \hline
	\end{tabular}
&

\begin{tabular}{|c|c|c|c|c|c|}
		\hline
	1&2	&1&	0&	0&	0\\
	\hline
0&1&	2&	1&	0&	0\\
  \hline
  0&1&	4&	3&	2&	1\\
  \hline
	\end{tabular}
&
\begin{tabular}{|c|c|c|c|c|c|}
		\hline
	1&2	&0&	0&	0&	0\\
	\hline
0&0&	0&	0&	0&	0\\
  \hline
  0&0&	3&	3&	2&	1\\
  \hline
	\end{tabular}
\end{array} $$
\end{flushleft}

$(4)\,\bullet \,\mg=(0,0,0,12,14+23)$:
	\begin{flushleft}	
\tiny $$\begin{array}{ccc}
E_0 & E_1 & E_2 \\
\\
\begin{tabular}{|c|c|c|c|c|c|}
	\hline
	1&3	&3&	1&	0&	0\\
	\hline
0&1&	3&	3&	1&	0\\
  \hline
  0&1&	4&	6&	4&	1\\
  \hline
	\end{tabular}
&

\begin{tabular}{|c|c|c|c|c|c|}
	\hline
	1&3	&3&	1&	0&	0\\
	\hline
0&1&	3&	3&	1&	0\\
  \hline
  0&1&	3&	4&3&	1\\
  \hline
	\end{tabular}
&
\begin{tabular}{|c|c|c|c|c|c|}
	\hline
	1&3	&2&	0&	0&	0\\
	\hline
0&0&	1&	1&	0&	0\\
  \hline
  0&0&	1&	3&	3&	1\\
  \hline
	\end{tabular}
\end{array} $$
\end{flushleft}

$(5)\,\bullet \,\mg=(0,0,0,12,24)$:  $\mg=\R e_3\oplus \mh$ where $\mh$ is the 3-step nilpotent Lie algebra of dimension 4.
	\begin{flushleft}	
\tiny $$\begin{array}{ccc}
E_0 & E_1 & E_2 \\
\\
\begin{tabular}{|c|c|c|c|c|c|}
	\hline
	1&3	&3&	1&	0&	0\\
	\hline
0&1&	3&	3&	1&	0\\
  \hline
  0&1&	4&	6&	4&	1\\
  \hline
	\end{tabular}
&

\begin{tabular}{|c|c|c|c|c|c|}
	\hline
	1&3	&3&	1&	0&	0\\
	\hline
  0&1&	3&	3&	1&	0\\
  \hline
  0&1&	3&	4&	3&	1\\
  \hline
	\end{tabular}
&
\begin{tabular}{|c|c|c|c|c|c|}
	\hline
	1&3&2&	0&	0&	0\\
	\hline
  0&0&1&	1&	0&	0\\
  \hline
  0&0&1&	3&	3&	1\\
  \hline
	\end{tabular}
\end{array} $$
\end{flushleft}

$(6)\,\bullet \,\mg=(0,0,0,12,13)$:  
	\begin{flushleft}	
\tiny $$\begin{array}{ccc}
E_0 & E_1 & E_2 \\
\\
\begin{tabular}{|c|c|c|c|c|c|}
	\hline
1&	3&	3&	1&0&0\\
  \hline
  0&2&	7&	9&	5&	1\\
  \hline
	\end{tabular}
&

\begin{tabular}{|c|c|c|c|c|c|}
	\hline
1&	3&	3&	1&0&0\\
  \hline
  0&2&	6&	6&	3&	1\\
  \hline
	\end{tabular}
&
\begin{tabular}{|c|c|c|c|c|c|}
	\hline
  1&3&	1&	0&0&	0\\
  \hline
  0&0&5&	6&	3&	1\\
  \hline
	\end{tabular}
\end{array} $$
\end{flushleft}

$(7)\,\bullet \,\mg=(0,0,0,0,12)$: $\mg=\R e_3\oplus\R e_4\oplus \mh_3$.
	\begin{flushleft}	
\tiny $$\begin{array}{ccc}
E_0 & E_1 & E_2 \\
\\
\begin{tabular}{|c|c|c|c|c|c|}
	\hline
	1&4	&6&	4&	1&	0\\
	\hline
0&1&	4&	6&	4&	1\\
  \hline
	\end{tabular}
&

\begin{tabular}{|c|c|c|c|c|c|}
	\hline
	1&4	&6&	4&	1&	0\\
	\hline
0&1&	4&	6&	4&	1\\
  \hline
	\end{tabular}
&
\begin{tabular}{|c|c|c|c|c|c|}
	\hline
	1&4	&5&	2&	0&	0\\
	\hline
0&0&	2&	5&	4&	1\\
  \hline
	\end{tabular}
\end{array} $$
\end{flushleft}

$(8)\,\bullet \,\mg=(0,0,0,0,12+34)$:
	\begin{flushleft}	
\tiny $$\begin{array}{ccc}
E_0 & E_1 & E_2 \\
\\
\begin{tabular}{|c|c|c|c|c|c|}
	\hline
	1&4	&6&	4&	1&	0\\
	\hline
0&1&	4&	6&	4&	1\\
  \hline
	\end{tabular}
&

\begin{tabular}{|c|c|c|c|c|c|}
	\hline
	1&4	&6&	4&	1&	0\\
	\hline
0&1&	4&	6&	4&	1\\
  \hline
	\end{tabular}
&
\begin{tabular}{|c|c|c|c|c|c|}
	\hline
	1&4	&5&	0&	0&	0\\
	\hline
0&0&	0&	5&	4&	1\\
  \hline
	\end{tabular}
\end{array} $$
\end{flushleft}

\pagebreak

{\bf \Large Dimension six}
\smallskip

$(1)\,\bullet \,\mg=(0,0,12,13,14+23,34+52)$: 					
\hspace{-2cm} 	\begin{flushleft}	
\tiny $$\begin{array}{cccc}
E_0 & E_1 & E_2 & E_3=E_\infty \\
\\
\begin{tabular}{|c|c|c|c|c|c|c|}
	\hline
	1&2	&1&	0&	0&	0&	0\\
	\hline
0&1&	2&	1&	0&	0&	0\\
\hline
0&1&	3&	3&	1&	0&	0\\
\hline
0&1&	4&	6&	4&	1&	0\\
\hline
0&1&	5&	10&	10&	5&	1\\
  \hline
	\end{tabular}
&

\begin{tabular}{|c|c|c|c|c|c|c|}
	\hline
1&2&1&0&0&0&0\\
\hline
0&1&2&1&0&0&0\\
\hline
0&1&2&2&1&0&0\\
\hline
0&1&2&2&2&1&0\\
\hline
0&1&2&3&3&2&1\\
  \hline
	\end{tabular}
&
\begin{tabular}{|c|c|c|c|c|c|c|}
	\hline
1&2&0&0&0&0&0\\
\hline
0&0&1&0&0&0&0\\
\hline
0&0&0&2&0&0&0\\
\hline
0&0&1&1&0&0&0\\
\hline
0&0&2&1&2&2&1\\
  \hline
	\end{tabular}
&
	\begin{tabular}{|c|c|c|c|c|c|c|}
	\hline
	1&2 & 0 & 0 & 0 & 0 & 0 \\
	\hline
	0&0 & 1 & 0 & 0 & 0 & 0 \\
	\hline
	0&0 & 0 & 0 & 0 & 0 & 0 \\
	\hline
	0&0 & 1 & 1 & 0 & 0 & 0 \\
	\hline
	0&0 & 0 & 1 & 2 & 2 & 1 \\
	\hline	
	\end{tabular}
\end{array} $$
\end{flushleft}
\normalsize
 		
$(2)\,\bullet \,\mg=(0,0,12,13,14,34+52)$:  		  					
 	\begin{flushleft}	
\tiny $$\begin{array}{cccc}
E_0 & E_1 & E_2 & E_3=E_\infty \\
\\
\begin{tabular}{|c|c|c|c|c|c|c|}
	\hline
	1&2	&1&	0&	0&	0&	0\\
	\hline
0&1&	2&	1&	0&	0&	0\\
\hline
0&1&	3&	3&	1&	0&	0\\
\hline
0&1&	4&	6&	4&	1&	0\\
\hline
0&1&	5&	10&	10&	5&	1\\
  \hline
	\end{tabular}
&

\begin{tabular}{|c|c|c|c|c|c|c|}
	\hline
1&2&1&0&0&0&0\\
\hline
0&1&2&1&0&0&0\\
\hline
0&1&2&2&1&0&0\\
\hline
0&1&2&2&2&1&0\\
\hline
0&1&2&3&3&5&1\\
  \hline
	\end{tabular}
&
\begin{tabular}{|c|c|c|c|c|c|c|}
	\hline
1&2&0&0&0&0&0\\
\hline
0&0&1&0&0&0&0\\
\hline
0&0&0&2&0&0&0\\
\hline
0&0&1&1&0&0&0\\
\hline
0&0&2&1&2&2&1\\
  \hline
	\end{tabular}
&
	\begin{tabular}{|c|c|c|c|c|c|c|}
	\hline
	1&2 & 0 & 0 & 0 & 0 & 0 \\
	\hline
	0&0 & 1 & 0 & 0 & 0 & 0 \\
	\hline
	0&0 & 0 & 0 & 0 & 0 & 0 \\
	\hline
	0&0 & 1 & 1 & 0 & 0 & 0 \\
	\hline
	0&0 & 0 & 1 & 2 & 2 & 1 \\
	\hline	
	\end{tabular}
\end{array} $$
\end{flushleft}
\normalsize

 		$(3)\,\bullet \,\mg=(0,0,12,13,14,15)$:  	\begin{flushleft}	
\tiny $$\begin{array}{cccc}
E_0 & E_1 & E_2 & E_3=E_\infty \\
\\
\begin{tabular}{|c|c|c|c|c|c|c|}
	\hline
	1&2	&1&	0&	0&	0&	0\\
	\hline
0&1&	2&	1&	0&	0&	0\\
\hline
0&1&	3&	3&	1&	0&	0\\
\hline
0&1&	4&	6&	4&	1&	0\\
\hline
0&1&	5&	10&	10&	5&	1\\
  \hline
	\end{tabular}
&

\begin{tabular}{|c|c|c|c|c|c|c|}
	\hline
1&2&1&0&0&0&0\\
\hline
0&1&2&1&0&0&0\\
\hline
0&1&2&2&1&0&0\\
\hline
0&1&2&2&2&1&0\\
\hline
0&1&2&3&3&2&1\\
  \hline
	\end{tabular}
&
\begin{tabular}{|c|c|c|c|c|c|c|}
	\hline
1&2&0&0&0&0&0\\
\hline
0&0&1&0&0&0&0\\
\hline
0&0&0&2&0&0&0\\
\hline
0&0&1&1&1&0&0\\
\hline
0&0&2&2&2&2&1\\
  \hline
	\end{tabular}
&
	\begin{tabular}{|c|c|c|c|c|c|c|}
	\hline
	1&2 & 0 & 0 & 0 & 0 & 0 \\
	\hline
	0&0 & 1 & 0 & 0 & 0 & 0 \\
	\hline
	0&0 & 0 & 1 & 0 & 0 & 0 \\
	\hline
	0&0 & 1 & 1 & 1 & 0 & 0 \\
	\hline
	0&0 &1  & 2 & 2 & 2 & 1 \\
	\hline	
	\end{tabular}
\end{array} $$
\end{flushleft}
\normalsize
 		
 	$(4)\,\bullet \,\mg=(0,0,12,13,14+23,15+24)$: 
 					
 	\begin{flushleft}	
\tiny $$\begin{array}{cccc}
E_0 & E_1 & E_2 & E_3=E_\infty \\
\\
\begin{tabular}{|c|c|c|c|c|c|c|}
	\hline
	1&2	&1&	0&	0&	0&	0\\
	\hline
0&1&	2&	1&	0&	0&	0\\
\hline
0&1&	3&	3&	1&	0&	0\\
\hline
0&1&	4&	6&	4&	1&	0\\
\hline
0&1&	5&	10&	10&	5&	1\\
  \hline
	\end{tabular}
&

\begin{tabular}{|c|c|c|c|c|c|c|}
	\hline
1&2&1&0&0&0&0\\
\hline
0&1&2&1&0&0&0\\
\hline
0&1&2&2&1&0&0\\
\hline
0&1&2&2&2&1&0\\
\hline
0&1&2&3&2&2&1\\
  \hline
	\end{tabular}
&
\begin{tabular}{|c|c|c|c|c|c|c|}
	\hline
1&2&0&0&0&0&0\\
\hline
0&0&1&0&0&0&0\\
\hline
0&0&0&2&0&0&0\\
\hline
0&0&1&1&1&0&0\\
\hline
0&0&2&2&2&2&1\\
  \hline
	\end{tabular}
&
	\begin{tabular}{|c|c|c|c|c|c|c|}
	\hline
	1&2 & 0 & 0 & 0 & 0 & 0 \\
	\hline
	0&0 & 1 & 0 & 0 & 0 & 0 \\
	\hline
	0&0 & 0 & 1 & 0 & 0 & 0 \\
	\hline
	0&0 & 1 & 1 & 1 & 0 & 0 \\
	\hline
	0&0 & 1& 2 & 2 & 2 & 1 \\
	\hline	
	\end{tabular}
\end{array} $$
\end{flushleft}
\normalsize

$(5)\,\bullet \,\mg=(0,0,12,13,14,15+23)$:  					
 	\begin{flushleft}	
\tiny $$\begin{array}{cccc}
E_0 & E_1 & E_2 & E_3=E_\infty \\
\\
\begin{tabular}{|c|c|c|c|c|c|c|}
	\hline
	1&2	&1&	0&	0&	0&	0\\
	\hline
0&1&	2&	1&	0&	0&	0\\
\hline
0&1&	3&	3&	1&	0&	0\\
\hline
0&1&	4&	6&	4&	1&	0\\
\hline
0&1&	5&	10&	10&	5&	1\\
  \hline
	\end{tabular}
&

\begin{tabular}{|c|c|c|c|c|c|c|}
	\hline
1&2&1&0&0&0&0\\
\hline
0&1&2&1&0&0&0\\
\hline
0&1&2&2&1&0&0\\
\hline
0&1&2&2&2&1&0\\
\hline
0&1&2&3&3&2&1\\
  \hline
	\end{tabular}
&
\begin{tabular}{|c|c|c|c|c|c|c|}
	\hline
1&2&0&0&0&0&0\\
\hline
0&0&1&0&0&0&0\\
\hline
0&0&0&2&0&0&0\\
\hline
0&0&1&1&1&0&0\\
\hline
0&0&2&2&2&2&1\\
  \hline
	\end{tabular}
&
	\begin{tabular}{|c|c|c|c|c|c|c|}
	\hline
	1&2 & 0 & 0 & 0 & 0 & 0 \\
	\hline
	0&0 & 1 & 0 & 0 & 0 & 0 \\
	\hline
	0&0 & 0 & 1 & 0 & 0 & 0 \\
	\hline
	0&0 & 1 & 1 & 1 & 0 & 0 \\
	\hline
	0&0 & 1  & 2 & 2 & 2 & 1 \\
	\hline	
	\end{tabular}
\end{array} $$
\end{flushleft}
\normalsize

$(6)\,\bullet \,\mg=(0,0,12,13,23,14)$: 
\begin{flushleft}
\tiny $$\begin{array}{ccc}
E_0 & E_1 & E_2=E_\infty \\
\\
\begin{tabular}{|c|c|c|c|c|c|c|}
	\hline
	1&2&	1&0&0&0&0\\
\hline
0&1&2&1&0&0&0\\
\hline
0&2&7&9&5&1&0\\
\hline
0&1&5&10&10&5&1 \\
  \hline
	\end{tabular}
&

\begin{tabular}{|c|c|c|c|c|c|c|}
	\hline
	1&2&1&0&0&0&0\\
	\hline
0&1&2&1&0&0&0\\
\hline
0&2&4&3&2&1&0\\
\hline
0&1&2&3&3&2&1\\
  \hline
	\end{tabular}
&
\begin{tabular}{|c|c|c|c|c|c|c|}
	\hline
1&2&0&0&0&0&0\\
\hline
0&0&0&0&0&0&0\\
\hline
0&0&2&3&2&0&0\\
\hline
0&0&2&3&2&2&1\\
  \hline
	\end{tabular}
\end{array} $$
\end{flushleft}
\normalsize

$(7)\,\bullet \,\mg=(0,0,12,13,23,14-25)$: \begin{flushleft}
\tiny $$\begin{array}{ccc}
E_0 & E_1 & E_2=E_\infty \\
\\
\begin{tabular}{|c|c|c|c|c|c|c|}
	\hline
	1&2&	1&0&0&0&0\\
\hline
0&1&2&1&0&0&0\\
\hline
0&2&7&9&5&1&0\\
\hline
0&1&5&10&10&5&1 \\
  \hline
	\end{tabular}
&

\begin{tabular}{|c|c|c|c|c|c|c|}
	\hline
	1&2&1&0&0&0&0\\
	\hline
0&1&2&1&0&0&0\\
\hline
0&2&4&3&2&1&0\\
\hline
0&1&2&3&3&2&1\\
  \hline
	\end{tabular}
&
\begin{tabular}{|c|c|c|c|c|c|c|}
	\hline
1&2&0&0&0&0&0\\
\hline
0&0&0&0&0&0&0\\
\hline
0&0&2&3&2&0&0\\
\hline
0&0&2&3&2&2&1\\
  \hline
	\end{tabular}
\end{array} $$
\end{flushleft}
\normalsize
\pagebreak
 		$(8)\,\bullet \,\mg=(0,0,12,13,23,14+25)$:  		
\begin{flushleft}
\tiny $$\begin{array}{ccc}
E_0 & E_1 & E_2=E_\infty \\
\\
\begin{tabular}{|c|c|c|c|c|c|c|}
	\hline
	1&2&	1&0&0&0&0\\
\hline
0&1&2&1&0&0&0\\
\hline
0&2&7&9&5&1&0\\
\hline
0&1&5&10&10&5&1 \\
  \hline
	\end{tabular}
&

\begin{tabular}{|c|c|c|c|c|c|c|}
	\hline
	1&2&1&0&0&0&0\\
	\hline
0&1&2&1&0&0&0\\
\hline
0&2&4&3&2&1&0\\
\hline
0&1&2&3&3&2&1\\
  \hline
	\end{tabular}
&
\begin{tabular}{|c|c|c|c|c|c|c|}
	\hline
1&2&0&0&0&0&0\\
\hline
0&0&0&0&0&0&0\\
\hline
0&0&2&3&2&0&0\\
\hline
0&0&2&3&2&2&1\\
  \hline
	\end{tabular}
\end{array} $$
\end{flushleft}
\normalsize

$(9)\,\bullet \,\mg=(0,0,0,12,14-23,15+34)$: \begin{flushleft}
\tiny $$\begin{array}{cccc}
E_0 & E_1 & E_2&E_3=E_\infty \\
\\
\begin{tabular}{|c|c|c|c|c|c|c|}
	\hline
	1&3&	3&1&0&0&0\\
\hline
0&1&3&3&1&0&0\\
\hline
0&1&4&6&4&1&0\\
\hline
0&1&5&10&10&5&1 \\
  \hline
	\end{tabular}
&

\begin{tabular}{|c|c|c|c|c|c|c|}
	\hline
1&3&3&1&0&0&0\\
	\hline
0&1&3&3&1&0&0\\
\hline
0&1&3&4&3&1&0\\
\hline
0&1&3&4&4&3&1\\
  \hline
	\end{tabular}
&
\begin{tabular}{|c|c|c|c|c|c|c|}
	\hline
1&3&2&0&0&0&0\\
\hline
0&0&1&1&0&0&0\\
\hline
0&0&0&2&1&0&0\\
\hline
0&0&2&2&3&3&1\\
  \hline
	\end{tabular}
&
\begin{tabular}{|c|c|c|c|c|c|c|}
	\hline
1&3&2&0&0&0&0\\
\hline
0&0&1&0&0&0&0\\
\hline
0&0&0&2&1&0&0\\
\hline
0&0&1&2&3&3&1\\
  \hline
	\end{tabular}
\end{array} $$
\end{flushleft}
\normalsize

$(10)\,\bullet \,\mg=(0,0,0,12,14,15+23)$: \begin{flushleft}
\tiny $$\begin{array}{ccc}
E_0 & E_1 & E_2=E_\infty \\
\\
\begin{tabular}{|c|c|c|c|c|c|c|}
	\hline
	1&3&	3&1&0&0&0\\
\hline
0&1&3&3&1&0&0\\
\hline
0&1&4&6&4&1&0\\
\hline
0&1&5&10&10&5&1 \\
  \hline
	\end{tabular}
&

\begin{tabular}{|c|c|c|c|c|c|c|}
	\hline
	1&3&3&1&0&0&0\\
	\hline
0&1&3&3&1&0&0\\
\hline
0&1&3&4&3&1&0\\
\hline
0&1&3&4&4&3&1\\
  \hline
	\end{tabular}
&
\begin{tabular}{|c|c|c|c|c|c|c|}
	\hline
1&3&2&0&0&0&0\\
\hline
0&0&1&1&0&0&0\\
\hline
0&0&0&2&2&0&0\\
\hline
0&0&2&3&3&3&1\\
  \hline
	\end{tabular}
\end{array} $$
\end{flushleft}
\normalsize

$(11)\,\bullet \,\mg=(0,0,0,12,14,15+23+24)$: 
\begin{flushleft}
\tiny $$\begin{array}{ccc}
E_0 & E_1 & E_2=E_\infty \\
\\
\begin{tabular}{|c|c|c|c|c|c|c|}
	\hline
	1&3&	3&1&0&0&0\\
\hline
0&1&3&3&1&0&0\\
\hline
0&1&4&6&4&1&0\\
\hline
0&1&5&10&10&5&1 \\
  \hline
	\end{tabular}
&

\begin{tabular}{|c|c|c|c|c|c|c|}
	\hline
	1&3&3&1&0&0&0\\
	\hline
0&1&3&3&1&0&0\\
\hline
0&1&3&4&3&1&0\\
\hline
0&1&3&4&4&3&1\\
  \hline
	\end{tabular}
&
\begin{tabular}{|c|c|c|c|c|c|c|}
	\hline
1&3&2&0&0&0&0\\
\hline
0&0&1&1&0&0&0\\
\hline
0&0&0&2&2&0&0\\
\hline
0&0&2&3&3&3&1\\
  \hline
	\end{tabular}
\end{array} $$
\end{flushleft}
\normalsize

$(12)\,\bullet \,\mg=(0,0,0,12,14,15+24)$: $\mg=\R e_3\oplus \mh$ where $\mh$ is the Lie algebra (1) of dimension 5. 

\begin{flushleft}
\tiny $$\begin{array}{ccc}
E_0 & E_1 & E_2=E_\infty \\
\\
\begin{tabular}{|c|c|c|c|c|c|c|}
	\hline
	1&3&	3&1&0&0&0\\
\hline
0&1&3&3&1&0&0\\
\hline
0&1&4&6&4&1&0\\
\hline
0&1&5&10&10&5&1 \\
  \hline
	\end{tabular}
&

\begin{tabular}{|c|c|c|c|c|c|c|}
	\hline
	1&3&3&1&0&0&0\\
	\hline
0&1&3&3&1&0&0\\
\hline
0&1&3&4&3&1&0\\
\hline
0&1&3&4&4&3&1\\
  \hline
	\end{tabular}
&
\begin{tabular}{|c|c|c|c|c|c|c|}
	\hline
1&3&2&0&0&0&0\\
\hline
0&0&1&1&0&0&0\\
\hline
0&0&0&2&2&0&0\\
\hline
0&0&2&3&3&3&1\\
  \hline
	\end{tabular}
\end{array} $$
\end{flushleft}
\normalsize

$(13)\,\bullet \,\mg=(0,0,0,12,14,15)$: $\mg=\R e_3\oplus \mh$ where $\mh$ is the Lie algebra (2) of dimension 5.

\begin{flushleft}
\tiny $$\begin{array}{ccc}
E_0 & E_1 & E_23=E_\infty \\
\\
\begin{tabular}{|c|c|c|c|c|c|c|}
	\hline
	1&3&	3&1&0&0&0\\
\hline
0&1&3&3&1&0&0\\
\hline
0&1&4&6&4&1&0\\
\hline
0&1&5&10&10&5&1 \\
  \hline
	\end{tabular}
&

\begin{tabular}{|c|c|c|c|c|c|c|}
	\hline
	1&3&3&1&0&0&0\\
	\hline
0&1&3&3&1&0&0\\
\hline
0&1&3&4&3&1&0\\
\hline
0&1&3&4&4&3&1\\
  \hline
	\end{tabular}
&
\begin{tabular}{|c|c|c|c|c|c|c|}
	\hline
1&3&2&0&0&0&0\\
\hline
0&0&1&1&0&0&0\\
\hline
0&0&0&2&2&0&0\\
\hline
0&0&2&3&3&3&1\\
  \hline
	\end{tabular}
\end{array} $$
\end{flushleft}
\normalsize

$(14)\,\bullet \,\mg=(0,0,0,12,13,14+35)$: 
\begin{flushleft}
\tiny $$\begin{array}{ccc}
E_0 & E_1 & E_2=E_\infty \\
\\
\begin{tabular}{|c|c|c|c|c|c|c|}
	\hline
1&3&3&1&0&0&0\\
\hline
0&2&7&9&5&1&0\\
\hline
0&1&5&10&10&5&1\\
  \hline
	\end{tabular}
&

\begin{tabular}{|c|c|c|c|c|c|c|}
	\hline
1&3&	3&	1&	0&	0&	0\\
\hline
0&2	&6&	6&	3&	1&	0\\
\hline
0&1&	3&	6&	6&	3&	1\\
  \hline
	\end{tabular}
&
\begin{tabular}{|c|c|c|c|c|c|c|}
	\hline
	1&3 & 1 & 0 & 0 & 0 &  0\\
	\hline
	0&0 & 4 & 3 & 0 & 0 & 0 \\
	\hline
	 0&0&0 & 3 & 5 & 3 & 1 \\
	\hline
	\end{tabular}
\end{array} $$
\end{flushleft}
\normalsize

$(15)\,\bullet \,\mg=(0,0,0,12,23,14+35)$: $\mg$ is the Lie algebra of upper $4\times 4$ real matrices.\begin{flushleft}
\tiny $$\begin{array}{ccc}
E_0 & E_1 & E_2=E_\infty \\
\\
\begin{tabular}{|c|c|c|c|c|c|c|}
	\hline
1&3&3&1&0&0&0\\
\hline
0&2&7&9&5&1&0\\
\hline
0&1&5&10&10&5&1\\
  \hline
	\end{tabular}
&

\begin{tabular}{|c|c|c|c|c|c|c|}
	\hline
1&3&	3&	1&	0&	0&	0\\
\hline
0&2	&6&	6&	3&	1&	0\\
\hline
0&1&	3&	6&	6&	3&	1\\
  \hline
	\end{tabular}
&
\begin{tabular}{|c|c|c|c|c|c|c|}
	\hline
	1&3 & 1 & 0 & 0 & 0 &  0\\
	\hline
	0&0 & 4 & 3 & 0 & 0 & 0 \\
	\hline
	 0&0& 0 & 3 & 5 & 3 & 1 \\
	\hline
	\end{tabular}
\end{array} $$
\end{flushleft}
\normalsize
		\pagebreak
		$(16)\,\bullet \,\mg=(0,0,0,12,23,14-35)$: 
\begin{flushleft}
\tiny $$\begin{array}{ccc}
E_0 & E_1 & E_2=E_\infty \\
\\
\begin{tabular}{|c|c|c|c|c|c|c|}
	\hline
1&3&3&1&0&0&0\\
\hline
0&2&7&9&5&1&0\\
\hline
0&1&5&10&10&5&1\\
  \hline
	\end{tabular}
&

\begin{tabular}{|c|c|c|c|c|c|c|}
	\hline
1&3&	3&	1&	0&	0&	0\\
\hline
0&2	&6&	6&	3&	1&	0\\
\hline
0&1&	3&	6&	6&	3&	1\\
  \hline
	\end{tabular}
&
\begin{tabular}{|c|c|c|c|c|c|c|}
	\hline
	1&3 & 1 & 0 & 0 & 0 &  0\\
	\hline
	0&0 & 4 & 3 & 0 & 0 & 0 \\
	\hline
	 0&0& 0& 3 & 5 & 3 & 1 \\
	\hline
	\end{tabular}
\end{array} $$
\end{flushleft}
\normalsize

    $(17)\,\bullet \,\mg=(0,0,0,12,14,24)$:$\mg=\R e_3\oplus \mh$ where $\mh$ is the Lie algebra (3) of dimension 5. 
 		\begin{flushleft}
\tiny $$\begin{array}{ccc}
E_0 & E_1 & E_2=E_\infty \\
\\
\begin{tabular}{|c|c|c|c|c|c|c|}
	\hline
1&3&3&1&0&0&0\\
\hline
0&1&3&3&1&0&0\\
\hline
0&2&9&16&14&6&1\\
  \hline
	\end{tabular}
&

\begin{tabular}{|c|c|c|c|c|c|c|}
	\hline
1&3&3&1&0&0&0\\
\hline
0&1&3&3&1&0&0\\
\hline
0&2&6&7&5&3&1\\
  \hline
	\end{tabular}
&
\begin{tabular}{|c|c|c|c|c|c|c|}
	\hline
	1&3 & 2 & 0 & 0 & 0 &  0\\
	\hline
	0&0 & 0 & 0 & 0 & 0 & 0 \\
	\hline
	 0&0& 3 & 6 & 5 & 3 & 1 \\
	\hline
	\end{tabular}

\end{array} $$
\end{flushleft}
\normalsize

	$(18)\,\bullet \,\mg=(0,0,0,12,13-24,14+23)$: 	
 		\begin{flushleft}
\tiny $$\begin{array}{ccc}
E_0 & E_1 & E_2=E_\infty \\
\\
\begin{tabular}{|c|c|c|c|c|c|c|}
	\hline
1&3&3&1&0&0&0\\
\hline
0&1&3&3&1&0&0\\
\hline
0&2&9&16&14&6&1\\
  \hline
	\end{tabular}
&

\begin{tabular}{|c|c|c|c|c|c|c|}
	\hline
1&3&3&1&0&0&0\\
\hline
0&1&3&3&1&0&0\\
\hline
0&2&6&7&5&3&1\\
  \hline
	\end{tabular}
&
\begin{tabular}{|c|c|c|c|c|c|c|}
	\hline
	1&3 & 2 & 0 & 0 & 0 &  0\\
	\hline
	0&0 & 0 & 0 & 0 & 0 & 0 \\
	\hline
	 0&0&3 & 6 & 5 & 3 & 1 \\
	\hline
	\end{tabular}

\end{array} $$
\end{flushleft}
\normalsize

$(19)\,\bullet \,\mg=(0,0,0,12,14,13-24)$: 	 		\begin{flushleft}
\tiny $$\begin{array}{ccc}
E_0 & E_1 & E_2=E_\infty \\
\\
\begin{tabular}{|c|c|c|c|c|c|c|}
	\hline
1&3&3&1&0&0&0\\
\hline
0&1&3&3&1&0&0\\
\hline
0&2&9&16&14&6&1\\
  \hline
	\end{tabular}
&

\begin{tabular}{|c|c|c|c|c|c|c|}
	\hline
1&3&3&1&0&0&0\\
\hline
0&1&3&3&1&0&0\\
\hline
0&2&6&7&5&3&1\\
  \hline
	\end{tabular}
&
\begin{tabular}{|c|c|c|c|c|c|c|}
	\hline
	1&3 & 2 & 0 & 0 & 0 &  0\\
	\hline
	0&0 & 0 & 0 & 0 & 0 & 0 \\
	\hline
	 0&0&3 & 6 & 5 & 3 & 1 \\
	\hline
	\end{tabular}

\end{array} $$
\end{flushleft}
\normalsize

$(20)\,\bullet \,\mg=(0,0,0,12,13+14,24)$: 	 		\begin{flushleft}
\tiny $$\begin{array}{ccc}
E_0 & E_1 & E_2=E_\infty \\
\\
\begin{tabular}{|c|c|c|c|c|c|c|}
	\hline
1&3&3&1&0&0&0\\
\hline
0&1&3&3&1&0&0\\
\hline
0&2&9&16&14&6&1\\
  \hline
	\end{tabular}
&

\begin{tabular}{|c|c|c|c|c|c|c|}
	\hline
1&3&3&1&0&0&0\\
\hline
0&1&3&3&1&0&0\\
\hline
0&2&6&7&5&3&1  \\
  \hline
	\end{tabular}
&
\begin{tabular}{|c|c|c|c|c|c|c|}
	\hline
1&3&2&0&0&0&0\\
\hline
0&0&0&0&0&0&0\\
\hline
0&0&3&6&5&3&1  \\
  \hline
	\end{tabular}
\end{array} $$
\end{flushleft}
\normalsize

$(21)\,\bullet \,\mg=(0,0,0,12,13,14+23)$: 		\begin{flushleft}
\tiny $$\begin{array}{ccc}
E_0 & E_1 & E_2=E_\infty \\
\\
\begin{tabular}{|c|c|c|c|c|c|c|}
	\hline
1&3&3&1&0&0&0\\
\hline
0&2&7&9&5&1&0\\
\hline
0&1&5&10&10&5&1\\
  \hline
	\end{tabular}
&

\begin{tabular}{|c|c|c|c|c|c|c|}
	\hline
1&3&3&1&0&0&0\\
\hline
0&2&6&6&3&1&0\\
\hline
0&1&3&6&6&3&1\\
  \hline
	\end{tabular}
&
\begin{tabular}{|c|c|c|c|c|c|c|}
	\hline
	1&3 & 1 & 0 & 0 & 0 &  0\\
	\hline
	0&0 & 4 & 4 & 1 & 0 & 0 \\
	\hline
	 0&0&1 & 4 & 5 & 3 & 1 \\
	\hline
	\end{tabular}
\end{array} $$
\end{flushleft}
\normalsize

$(22)\,\bullet \,\mg=(0,0,0,12,13,24)$: 			\begin{flushleft}
\tiny $$\begin{array}{ccc}
E_0 & E_1 & E_2=E_\infty \\
\\
\begin{tabular}{|c|c|c|c|c|c|c|}
	\hline
1&3&3&1&0&0&0\\
\hline
0&2&7&9&5&1&0\\
\hline
0&1&5&10&10&5&1\\
  \hline
	\end{tabular}
&

\begin{tabular}{|c|c|c|c|c|c|c|}
	\hline
1&3&3&1&0&0&0\\
\hline
0&2&6&6&3&1&0\\
\hline
0&1&3&6&6&3&1\\
  \hline
	\end{tabular}
&
\begin{tabular}{|c|c|c|c|c|c|c|}
	\hline
	1&3 & 1 & 0 & 0 & 0 &  0\\
	\hline
	0&0 & 4 & 4 & 1 & 0 & 0 \\
	\hline
	 0&0&1 & 4 & 5 & 3 & 1 \\
	\hline
	\end{tabular}
\end{array} $$
\end{flushleft}
\normalsize

$(23)\,\bullet \,\mg=(0,0,0,12,13,14)$: 
 		\begin{flushleft}
\tiny $$\begin{array}{ccc}
E_0 & E_1 & E_2 =E_\infty \\
\\
\begin{tabular}{|c|c|c|c|c|c|c|}
	\hline
1&3&3&1&0&0&0\\
\hline
0&2&7&9&5&1&0\\
\hline
0&1&5&10&10&5&1\\
  \hline
	\end{tabular}
&

\begin{tabular}{|c|c|c|c|c|c|c|}
	\hline
1&3&3&1&0&0&0\\
\hline
0&2&6&6&3&1&0\\
\hline
0&1&3&6&6&3&1\\
  \hline
	\end{tabular}
&
\begin{tabular}{|c|c|c|c|c|c|c|}
	\hline
	1&3 & 1 & 0 & 0 & 0 &  0\\
	\hline
	0&0 & 4 & 4 & 1 & 0 & 0 \\
	\hline
	 0&0&1 & 4 & 5 & 3 & 1 \\
	\hline
	\end{tabular}
\end{array} $$
\end{flushleft}
\normalsize

$(24)\,\bullet \,\mg=(0,0,0,12,13,23)$: 
\begin{flushleft}
\tiny $$\begin{array}{ccc}
E_0 & E_1 & E_2=E_\infty \\
\\
\begin{tabular}{|c|c|c|c|c|c|c|}
	\hline
	1&3&3&1&0&0&0\\
	\hline
0&3&12&19&15&6&1\\
  \hline
	\end{tabular}
&

\begin{tabular}{|c|c|c|c|c|c|c|}
	\hline
	1&3&3&1&0&0&0\\
	\hline
0&3&9&12&8&3&1\\
  \hline
	\end{tabular}
&
	\begin{tabular}{|c|c|c|c|c|c|c|}
	\hline
	1&3 & 0 & 0 & 0 & 0 &  0\\
	\hline
	0&0 & 8 & 12& 8 & 3 & 1 \\
	\hline
	\end{tabular}
\end{array} $$
\end{flushleft}
\normalsize
\pagebreak
$(25)\,\bullet \,\mg=(0,0,0,0,12,15+34)$: 

 		\begin{flushleft}
\tiny $$\begin{array}{cccc}
E_0 & E_1 & E_2& E_3=E_\infty \\
\\
\begin{tabular}{|c|c|c|c|c|c|c|}
	\hline
	1&4&6&4&1&0&0\\
	\hline
0&1&4&6&4&1&0\\
\hline
0&1&4&10&10&4&1\\
  \hline
	\end{tabular}
&

\begin{tabular}{|c|c|c|c|c|c|c|}
	\hline
	1&4&6&4&1&0&0\\
	\hline
0&1&4&6&4&1&0\\
\hline
0&1&4&7&7&4&1\\
  \hline
	\end{tabular}
&
	\begin{tabular}{|c|c|c|c|c|c|c|}
	\hline
1&4&5&2&0&0&0\\
\hline
0&0&1&2&0&0&0\\
\hline
0&0&1&3&6&4&1\\
\hline
	\end{tabular}

&
	\begin{tabular}{|c|c|c|c|c|c|c|}
	\hline
	1&4 & 5 & 1 & 0 & 0 &  0\\
	\hline
	0&0 & 1 & 2 & 0 & 0 & 0 \\
	\hline
	 0&0& 0 & 3 & 6 & 4 & 1 \\
	\hline
	\end{tabular}

\end{array} $$
\end{flushleft}
\normalsize

$(26)\,\bullet \,\mg=(0,0,0,0,12,15)$: $\mg=\R e_3\oplus \R e_4\oplus \mh$ where $\mh$ is the 3-step nilpotent Lie algebra of dimension four.
 		\begin{flushleft}
\tiny $$\begin{array}{ccc}
E_0 & E_1 & E_2=E_\infty \\
\\
\begin{tabular}{|c|c|c|c|c|c|c|}
	\hline
	1&4&6&4&1&0&0\\
	\hline
0&1&4&6&4&1&0\\
\hline
0&1&5&10&10&4&1\\
  \hline
	\end{tabular}
&

\begin{tabular}{|c|c|c|c|c|c|c|}
	\hline
	1&4&6&4&1&0&0\\
	\hline
0&1&4&6&4&1&0\\
\hline
0&1&4&7&7&4&1\\
  \hline
	\end{tabular}
&
	\begin{tabular}{|c|c|c|c|c|c|c|}
	\hline
1&4&5&2&0&0&0\\
\hline
0&0&1&2&1&0&0\\
\hline
0&0&1&4&6&4&1\\
\hline
	\end{tabular}

\end{array} $$
\end{flushleft}
\normalsize


$(27)\,\bullet \,\mg=(0,0,0,0,12,14+25)$:  		\begin{flushleft}
\tiny $$\begin{array}{ccc}
E_0 & E_1 & E_2=E_\infty \\
\\
\begin{tabular}{|c|c|c|c|c|c|c|}
	\hline
	1&4&6&4&1&0&0\\
	\hline
0&1&4&6&4&1&0\\
\hline
0&1&5&10&10&5&1\\
  \hline
	\end{tabular}
&

\begin{tabular}{|c|c|c|c|c|c|c|}
	\hline
	1&4&6&4&1&0&0\\
	\hline
0&1&4&6&4&1&0\\
\hline
0&1&4&7&7&4&1\\
  \hline
	\end{tabular}
&
	\begin{tabular}{|c|c|c|c|c|c|c|}
	\hline
1&4&5&2&0&0&0\\
\hline
0&0&1&2&1&0&0\\
\hline
0&0&1&4&6&4&1\\
\hline
	\end{tabular}

\end{array} $$
\end{flushleft}
\normalsize

 		$(28)\,\bullet \,\mg=(0,0,0,0,13+42,14+23)$: 

\begin{flushleft}
\tiny $$\begin{array}{ccc}
E_0 & E_1 & E_2=E_\infty \\
\\
\begin{tabular}{|c|c|c|c|c|c|c|}
	\hline
	1&4&6&4&1&0&0\\
	\hline
0&2&9&16&14&6&1\\
  \hline
	\end{tabular}
&

\begin{tabular}{|c|c|c|c|c|c|c|}
	\hline
	1&4&6&4&1&0&0\\
	\hline
0&2&8&11&8&4&1\\
  \hline
	\end{tabular}
&
	\begin{tabular}{|c|c|c|c|c|c|c|}
	\hline
	1&4 & 4 & 0 & 0 & 0 &  0\\
	\hline
	0&0 & 4 & 10& 8 & 4 & 1 \\
	\hline
	\end{tabular}
\end{array} $$
\end{flushleft}
\normalsize

	$(29)\,\bullet \,\mg=(0,0,0,0,12,14+23)$:
\begin{flushleft}
\tiny $$\begin{array}{ccc}
E_0 & E_1 & E_2=E_\infty \\
\\
\begin{tabular}{|c|c|c|c|c|c|c|}
	\hline
	1&4&6&4&1&0&0\\
	\hline
0&2&9&16&14&6&1\\
  \hline
	\end{tabular}
&

\begin{tabular}{|c|c|c|c|c|c|c|}
	\hline
	1&4&6&4&1&0&0\\
	\hline
0&2&8&11&8&4&1\\
  \hline
	\end{tabular}
&
	\begin{tabular}{|c|c|c|c|c|c|c|}
	\hline
	1&4 & 4 & 0 & 0 & 0 &  0\\
	\hline
	0&0 & 4 & 10& 8 & 4 & 1 \\
	\hline
	\end{tabular}
\end{array} $$
\end{flushleft}
\normalsize

 		$(30)\,\bullet \,\mg=(0,0,0,0,12,34)$: 
\begin{flushleft}
\tiny $$\begin{array}{ccc}
E_0 & E_1 & E_2=E_\infty \\
\\
\begin{tabular}{|c|c|c|c|c|c|c|}
	\hline
	1&4&6&4&1&0&0\\
	\hline
0&2&9&16&14&6&1\\
  \hline
	\end{tabular}
&

\begin{tabular}{|c|c|c|c|c|c|c|}
	\hline
	1&4&6&4&1&0&0\\
	\hline
0&2&8&11&8&4&1\\
  \hline
	\end{tabular}
&
	\begin{tabular}{|c|c|c|c|c|c|c|}
	\hline
	1&4 & 4 & 0 & 0 & 0 &  0\\
	\hline
	0&0 & 4 & 10& 8 & 4 & 1 \\
	\hline
	\end{tabular}
\end{array} $$
\end{flushleft}
\normalsize		

$(31)\,\bullet \,\mg=(0,0,0,0,12,13)$: $\mg=\R e_4\oplus \mh$ where $\mh$ is the Lie algebra (6) of dimension 5.

\begin{flushleft}
\tiny $$\begin{array}{ccc}
E_0 & E_1 & E_2=E_\infty \\
\\
\begin{tabular}{|c|c|c|c|c|c|c|}
	\hline
	1&4&6&4&1&0&0\\
	\hline
0&2&9&16&14&6&1\\
  \hline
	\end{tabular}
&

\begin{tabular}{|c|c|c|c|c|c|c|}
	\hline
	1&4&6&4&1&0&0\\
	\hline
0&2&8&12&9&4&1\\
  \hline
	\end{tabular}
&
	\begin{tabular}{|c|c|c|c|c|c|c|}
	\hline
	1&4 & 4 & 1 & 0 & 0 &  0\\
	\hline
	0&0 &5& 11& 9 & 4 & 1 \\
	\hline
	\end{tabular}
\end{array} $$
\end{flushleft}
\normalsize		

 		$(32)\,\bullet \,\mg=(0,0,0,0,0,12+34)$: $\mg=\R e_5\oplus \mh$ where $\mh$ is the Lie algebra (8) of dimension 5. 
 		\begin{flushleft}
\tiny $$\begin{array}{ccc}
E_0 & E_1 & E_2=E_\infty \\
\\
\begin{tabular}{|c|c|c|c|c|c|c|}
	\hline
1&5&10&10&5&1&0\\
\hline
0&1&5&10&10&5&1\\
  \hline
	\end{tabular}
&

\begin{tabular}{|c|c|c|c|c|c|c|}
	\hline
1&5&10&10&5&1&0\\
\hline
0&1&5&10&10&5&1\\
  \hline
	\end{tabular}
&
	\begin{tabular}{|c|c|c|c|c|c|c|}
	\hline
	1&5 &9 & 5 & 0 & 0 &  0\\
	\hline
	0&0 & 0 & 5& 9 & 5 & 1 \\
	\hline
	\end{tabular}
\end{array} $$
\end{flushleft}
\normalsize

$(33)\,\bullet \,\mg=(0,0,0,0,0,12)$:$\mg=\R e_3\oplus\R e_4\oplus \mh$ where $\mh$ is the nilpotent Lie algebra or dimension 3.
\begin{flushleft}
\tiny $$\begin{array}{ccc}
E_0 & E_1 & E_2=E_\infty \\
\\
\begin{tabular}{|c|c|c|c|c|c|c|}
	\hline
	1&5&10&10&5&1&0\\
	\hline
0&1&5&10&10&5&1\\
  \hline
	\end{tabular}
&

\begin{tabular}{|c|c|c|c|c|c|c|}
	\hline
	1&5&10&10&5&1&0\\
	\hline
0&1&5&10&10&5&1\\
  \hline
	\end{tabular}
&
	\begin{tabular}{|c|c|c|c|c|c|c|}
	\hline
	1&5 & 9 & 7 & 2 & 0 &  0\\
	\hline
	0&0 & 2 & 7 & 9 & 5 & 1 \\
	\hline
	\end{tabular}
\end{array} $$
\end{flushleft}
\normalsize		

\bigskip
\bigskip

From simple inspection of the tables of the limit of the spectral sequence showed above we can see that in dimension five there are 6 different configurations of tables that correspond to 8 isomorphisms classes of nilpotent Lie algebras in that dimension. Meanwhile, there are 15 diffe\-rent configurations of tables that correspond to 33 isomorphisms classes of nilpotent Lie algebras in dimension six. Thus, there are non-isomorphic nilpotent Lie algebras with the same table. Nevertheless, notice that the Lie algebras 16 and 17 of dimension six have the same Betti numbers but different tables.

\end{document}